\documentclass[a4paper,11pt]{elsarticle}
\usepackage{amsmath,amssymb,amsfonts}
\usepackage{amsthm}
\usepackage{graphicx,color}
\usepackage{mathtools}
\newtheorem{lemma}{Lemma}

\newtheorem{theorem}{Theorem}

\numberwithin{equation}{section}

\journal{Applied Numerical Mathematics}

\begin{document}
\nocite{*}

\begin{frontmatter}
\title{Mean square asymptotic stability characterisation of perturbed linear stochastic functional differential equations}

\author{John A. D. Appleby$^a$, Emmet Lawless$^a$}

\affiliation{School of Mathematical Sciences, Dublin City University.}

\ead{john.appeby@dcu.ie, emmet.lawless6@mail.dcu.ie}

\begin{abstract}
In this paper we investigate the mean square asymptotic stability of a perturbed scalar linear stochastic functional differential equation. Specifically, we are able to give necessary and sufficient conditions on the forcing terms for convergence of the mean square, exponential convergence of the mean square, and integrability of the mean square of solutions. It is also essential that the underlying unperturbed SFDE is mean--square asymptotically stable for these results to hold.\\
\end{abstract}

\end{frontmatter}

\section{Introduction}

Over the last three decades, a substantial literature has been developed concerning the asymptotic behaviour and stability properties of stochastic functional differential equations (SFDEs). Important monographs by Kolmanovskii and Myshkis~\cite{KolMys}, Mao~\cite{Mao} and Shaikhet~\cite{Shaikhet2013} have appeared. A variety of stability types can be considered, but from the outset of stability studies in stochastic functional differential (or evolution) equations the asymptotic behaviour, and specifically the convergence in the mean square, has attracted a great deal of attention (see for instance the Haussmann~\cite{Hauss:78}, Ichikawa~\cite{Ichi:82}, Mao~\cite{Mao:94}, Mizel and Trutzer~\cite{MizTrut}, Mohammed~\cite{Moh:84}, as well as in~\cite{KolMys,Mao,Shaikhet2013} where  substantial bibliographies can be found). Moreover, since linear equations are so fundamental to mathematical analysis of hereditary systems, it is appropriate that they receive due study. 

The simplest class of equations that one might study are therefore (unforced) linear autonomous SFDEs. A characterisation of the mean--square behaviour of general scalar linear SFDEs (with finite memory) was produced in Appleby, Riedle and Mao~\cite{AMR}. In particular if $B$ is a standard Brownian motion, for the scalar equation 
\[
dU(t)=\int_{[-\tau,0]}U(t+s)\nu(ds)\,dt + \int_{[-\tau,0]} U(t+s)\mu(ds)\,dB(t),
\]
with continuous and deterministic initial function $\psi$, where $\mu$ and $\nu$ are finite measures on $[-\tau,0]$, it has been shown that the global mean square asymptotic stability of the zero solution of this equation is equivalent to 
\[
r(t)\to 0 \text{ as $t\to\infty$}, \quad \int_0^\infty \left(\int_{[-\tau,0]} r(t+s)\mu(ds)\right)^2\,dt <  1,
\] 
where $r$ is the differential resolvent of the underlying deterministic differential equation i.e.,
\[
r'(t)=\int_{[-\tau,0]}r(t+s)\nu(ds), \quad t>0; \quad r(0)=1, \quad r(t)=0 \quad\text{$t\in [-\tau,0)$}.
\]
Indeed, if the mean square convergence to zero occurs, it must do so exponentially fast. There is a very extensive literature on exponential mean square stability  for SFDEs with finite delay,  with great advances being made by Mao and co--workers: some representative and foundational works from this school include \cite{Mao:1996,Mao:2001,MaoLiao:1996,MaoShah:1997} and the themes of this work and more recent advances is reflected in the monograph~\cite{Mao}. Other papers which seek to give a characterisation of the mean square asymptotic stability of solutions of unforced SFDEs include \cite{bucknot2013}, \cite{MacNech} and \cite{app21}. A corresponding characterisation for linear stochastic difference equations is presented by the authors in \cite{AL23}.

It is very natural to then ask: if the (exponentially stable) equation is perturbed by external forces independent of the state which are, in a certain sense, asymptotically small, what are the minimal conditions on these forcing terms that preserve the stability (in various ways)? Such forcing terms are sometimes called fading perturbations or damped noise. Concretely, if we put deterministic and continuous functions $f$ and $g$ in the drift and diffusion terms, to get the forced equation
\[
dX(t)=\left(f(t)+\int_{[-\tau,0]}X(t+s)\nu(ds)\right)\,dt + \left(g(t)+ \int_{[-\tau,0]} X(t+s)\mu(ds)\right)\,dB(t),
	\]
	under what conditions does the mean square still tend to zero, or do so exponentially fast, or do so in an integrable sense i.e.,
	\[
	\int_0^\infty \mathbb{E}[X^2(t,\psi)]\,dt <+\infty?
	\]
Our results in this paper (for the scalar equation at least) are able to characterise exactly the conditions which give rise to the various types of mean square convergence. In particular, we are able to show that $\mathbb{E}[X^2(t,\psi)]\to 0$ as $t\to\infty$ for every continuous initial condition $\psi$ with finite second moments is \textit{equivalent to} the mean square convergence of $U$ to zero (i.e. the stochastic equation without any forcing), alongside the conditions 
\[
\int_t^{t+\delta} f(s)\,ds\to 0, \quad\text{$t\to\infty$ for each $\delta\in (0,1]$}, 
\quad \int_t^{t+1} g^2(s)\,ds \to 0 \quad\text{$t\to\infty$}.
\] 
If exponential convergence of $\mathbb{E}[X^2(t)]$ to zero is desired as $t\to\infty$, this is equivalent to the mean square convergence of $U$ to zero, alongside the following exponential decay conditions on $f$ and $g$: 
\[
\text{There is $C>0$, $\beta>0$ such that }
\left|\int_0^t e^{\beta s}f(s)\,ds\right|\leq C, \quad t\geq 0,\quad \int_0^\infty e^{2\beta s}g^2(s)\,ds < +\infty.
\]
Finally, if we want mean square integrability, as characterised above, this is equivalent to the mean square convergence of $U$ to zero, alongside the following square integrability conditions fulfilled by $f$ and $g$:
\[t\mapsto 
\int_0^t e^{-(t-s)}f(s)\,ds \in L^2(\mathbb{R}_+), \quad g\in L^2(\mathbb{R}_+).
\] 
In the opinion of the authors, these results constitute a solid advance in the theory, since \textit{coincident necessary and sufficient conditions on $f$ and $g$ are imposed which guarantee the appropriate type of convergence}. Until now, sufficient condition results abound, but such an exact characterisation has not been achieved. It is to be noted that these conditions do not place pointwise bounds on $f$ and $g$, but are rather conditions on certain types of averages of $f$ and $g$. Such conditions allow for relatively ill--behaved forcing functions on a pointwise basis, provided their ``average'' behaviour is good. On the other hand, the result has the character of deterministic perturbation theorems: if the underlying unperturbed equation (in this case $U$) has asymptotically stable solutions, then so does the forced equation, contingent on the forcing function fading sufficiently rapidly. Other results achieve something of this goal, but here we are able to exactly characterise the underlying stability condition, as well as the precise conditions on the forcing term which enable the result to hold. The exponential integral properties noted above have occurred already in the literature for \textit{affine} stochastic Volterra integrodifferential equations (see e.g.~\cite{AppFree:03, Mao:2000}), in which there is no state dependence in the diffusion term. For deterministic Volterra equations, an application of \cite[Thm 11.4.3]{GLS} in the linear case shows that the condition 
\begin{equation} \label{eq.fto0ave}
\int_t^{t+\delta} f(s)\,ds \to 0 \text{ as $t\to\infty$ for all $\delta\in (0,1]$},
\end{equation}
is sufficient for the perturbed equation 
\[
x'(t) =\int_{[0,t]}\mu(ds)x(t-s) + f(t), \quad 
\] 
to obey $x(t)\to 0$ as $t\to\infty$, provided the underlying differential resolvent $r$, given by 
\[
r'(t)=\int_{[0,t]}\mu(ds)r(t-s), \quad t>0; \quad r(0)=1,
\]
is in $L^1(\mathbb{R}_+)$. The result underlying \eqref{eq.fto0ave} appears as \cite[Lemma 15.9.2]{GLS}, and relies on an elegant decomposition of $f$ into components which depend solely on the ``sectional average'' in \eqref{eq.fto0ave}. Specifically, extending $f$ to be zero on $(-\infty,0)$ and writing 
\[
f_\delta(t)=\int_{t-\delta}^t f(s)\,ds, \quad t\geq 0
\]
we have for any $\delta\in (0,1)$ that 
\[
I(t):=f(t)-f_1(t), \quad t\geq 0, 
\]
obeys 
\[
\int_0^t I(s)\,ds = \int_0^1f_\delta(t)\,d\delta, \quad t\geq 1.
\]
This identity enables conditions to be imposed on the more regular functions $f_\delta$ than on $f$ directly, and, in the case of this work at least, lead to sharp characterisations of stability. Moreover, in the opinion of the authors, this decomposition lemma can make further contributions to characterise the asymptotic behaviour in deterministic and stochastic equations with memory, and we hope to explore this in further works.

The condition \eqref{eq.fto0ave} goes back further than the late 1980's, however. To the best of our knowledge, one of the earliest uses of this condition in the stability theory of asymptotically autonomous deterministic differential equations is in papers of Strauss and Yorke~\cite{SY67a,SY67b} in the late 1960's. However, as far as we know, in the present work we see the first use in stochastic equations in characterising asymptotic stability. On the other hand, the importance of the average  $\int_t^{t+1}g^2(s)\,ds$ in the diffusion term of SDEs was first pointed out in Chan and Williams~\cite{CW89}, with results which eliminate monotonicity in $g$ being presented in Appleby, Cheng and Rodkina~\cite{ACR11}.

The paper is organised as follows: Section 2 gives a precise formulation of the problem, together with some background theory. Section 3 deduces linear Volterra integral equations for the mean square of $X$, as well as some auxiliary functions (which are mean--squares of other processes). Section 4 states the main results and discusses hypotheses. Section 5, which concludes the paper, gives the proofs of the main results.  
  
\section{Mathematical Preliminaries}
For the following general results on SFDEs and stochastic analysis, the reader may refer to the monographs~\cite{Mao} and~\cite{KolMys}. 
Let us fix a complete probability space $(\Omega,\mathcal{F},\mathbb{P})$ with a
filtration $(\mathcal{F}(t))_{t\ge 0}$ satisfying the usual conditions and
let $B=\{B(t):\,t\geq 0\}$ be a one--dimensional Brownian motion on this space. We denote by $\mathbb{E}$ the expectation induced from $\mathbb{P}$: that is to say, for any $\mathcal{F}$--measurable $X$ (i.e. a mapping  $X:\Omega\to\mathbb{R}$ such that $\{X\leq x\}\in\mathcal{F}$ for all $x\in \mathbb{R}$) for which 
\[
\int_{\Omega} |X(\omega)|\,d\mathbb{P}[\{\omega\}]<+\infty,
\]
we have 
\[
\mathbb{E}[X] = \int_{\Omega} X(\omega)\,d\mathbb{P}[\{\omega\}]. 
\]

Let $\tau>0$. A process $X=\{X(t):t\geq 0\}$ is said to be adapted to  $(\mathcal{F}(t))_{t\ge 0}$ if 
$X(t)$ is $\mathcal{F}(t)$--measurable for each $t\geq 0$. 
Let $\psi$ be a $C([-\tau,0];\mathbb{R})$--valued $\mathcal{F}(0)$--measurable random variable with  
\[
\|\psi\|^2:=
\mathbb{E}\left[  \sup_{t\in [-\tau,0]} \psi^2(t)\right] < +\infty,
\]
recalling that $C([-\tau,0];\mathbb{R})$ is the space of continuous functions $\varphi:[-\tau,0]\to\mathbb{R}$ equipped with the norm $|\varphi|=\sup_{t\in[-\tau,0]} |\varphi(t)|$. We assume also 
\begin{equation} \label{eq.second moment psi}
\|\psi\|<+\infty.
\end{equation}
$\psi$ is also assumed to be independent of the Brownian motion $B$. 
Note that the finiteness of $\|\psi\|$ and the Dominated Convergence Theorem ensure that $t\mapsto \mathbb{E}[\psi^2(t)]$ is continuous on $[-\tau,0]$. Denote
\begin{equation} \label{def.phi}
\phi(t):=\sqrt{\mathbb{E}[\psi^2(t)]}, \quad t\in [-\tau,0],
\end{equation} 
so that $\phi^2(t)=\mathbb{E}[\psi^2(t)]$ for $t\in [-\tau,0]$. Both $\phi$ and $\phi^2$ are continuous. 

 We first state the unperturbed equation whose asymptotic behaviour is of paramount importance. If we let $\nu$ and $\mu$ be in $ M[-\tau,0]$, the space of finite Borel measures on $[-\tau,0]$, the unperturbed equation has the following form:
\begin{align} \label{eq. Unperturbed SFDE}
    dU(t) & = \left(\int_{[-\tau,0]}U(t+s)\nu(ds)\right)dt + \left(\int_{[-\tau,0]}U(t+s)\mu(ds)\right)dB(t), \quad t \geq 0\\
    U(t) & = \psi(t), \quad t\leq 0, \nonumber
\end{align}
where $\psi$ has the properties indicated above. \eqref{eq. Unperturbed SFDE} is so--called differential shorthand for 
\begin{align*}
 U(t) & = U(0)+\int_0^t \left(\int_{[-\tau,0]}U(s+u)\nu(du)\right)ds + \int_0^t \left(\int_{[-\tau,0]}U(s+u)\mu(du)\right)dB(s), \quad t \geq 0\\
U(t) & = \psi(t), \quad t\leq 0. \nonumber
\end{align*}
For every $\psi$ as specified above, there exists a unique, continuous, adapted
(to $(\mathcal{F}(t))_{t\geq 0}$) process $U=\{U(t,\psi):\,t\geq -\tau\}$ which satisfies \eqref{eq. Unperturbed SFDE}. The process is unique in the following sense: if there is another continuous, adapted process $\tilde{U}$ which satisfies \eqref{eq. Unperturbed SFDE}, then 
\[
\mathbb{P}[U(t)=\tilde{U}(t) \text{ for all $t\geq 0$}]=1.
\] 
This process $X$ is a so--called strong solution of \eqref{eq. Unperturbed SFDE}, and $U$ has finite second moments (cf., e.g., Mao~\cite[Theorem 5.2.7]{Mao}). This means that 
\[
\mathbb{E}[U^2(t)]<+\infty, \quad \text{ for all $t\geq -\tau$},
\]
and, a fortiori, 
\[
\mathbb{E}\left[\sup_{-\tau\leq s\leq t} U^2(s)\right] <+\infty, \quad \text{ for all $t\geq -\tau$}.
\]

The equation \eqref{eq. Unperturbed SFDE} was studied extensively by 
Appleby et al.~\cite{AMR} in which they gave a full characterisation of the mean square behaviour of \eqref{eq. Unperturbed SFDE}. This includes a set of necessary and sufficient conditions which ensures $\mathbb{E}[U^2(t,\psi)] \to 0$ as $t \to \infty$ for all initial functions $\psi$. It should be noted in \cite{AMR} the authors only considered deterministic initial functions: however with the additional assumption \eqref{eq.second moment psi},  
\textit{a condition which we will impose throughout this paper without further reference in our results}, all of their results carry over to the case of random initial functions. In  \cite{AMR} it was found that the stochastic stability is heavily dependent on the behaviour of the underlying deterministic equation and moreover the fundamental resolvent. In this paper we demonstrate that this still prevails and so we introduce both of these objects in detail. The deterministic\footnote{To use the term \emph{deterministic} to describe equation \eqref{eq. Deterministic x0} is technically incorrect due to the presence of the random initial function. However its dynamics are indeed deterministic so it is in this spirit that we will continue to refer to equation \eqref{eq. Deterministic x0} as \emph{deterministic}.}, unperturbed delay equation is given by,
\begin{align} \label{eq. Deterministic x0}
    \dot{x}_0(t)  & =  \int_{[-\tau,0]}x_0(t+u)\nu(du), \quad t\geq0,\\
    x_0(t) & = \psi(t), \quad t \in [-\tau,0], \nonumber
\end{align}
where both $\nu$ and $\psi$ are defined as above. There are many texts that deal with deterministic delay equations, for further analysis of \eqref{eq. Deterministic x0} we refer the reader to \cite{DGLW,HL}. The underlying integral resolvent is the unique locally absolutely continuous function $r:[0,\infty)\to \mathbb{R}$ which satisfies
\begin{equation} \label{eq. resolvent}
    r(t) = 1+\int_0^t \int_{[\max\{-\tau,-s\},0]}r(s+u)\nu(du)ds, \quad t\geq0. 
\end{equation}
The above equation can be written in differential form by specifying $r(0)=1$ and $r(t)=0$ for all $t<0$. 
In all results throughout this paper we need to make assumptions on the asymptotic behaviour of the resolvent $r$ and its connection with the measure $\nu$. The following description of the asymptotic behaviour is standard and may be found in~\cite{DGLW,HL}. As pointed out in Appleby et al.~\cite{AMR}, the following conditions on solutions to \eqref{eq. resolvent} are all equivalent:
\begin{itemize}
    \item[(a)]$r(t)\to 0$, as $t\to \infty;$
    \item[(b)] $r \in L^1(\mathbb{R}_+);$
    \item[(c)] $r \in L^2(\mathbb{R}_+)$.
\end{itemize}
Henceforth the above relations will be used interchangeably without reference. An important detail regarding the stability of the resolvent is that whenever any of the above conditions are fulfilled $r$ tends to zero \emph{exponentially} fast. To see this, one looks for solutions to \eqref{eq. resolvent} of exponential type which leads to a transcendental characteristic equation. For all $\lambda \in \mathbb{C}$ we may define 
\[
h(\lambda) = \lambda - \int_{[-\tau,0]} e^{\lambda s}\nu(ds).
\]
We pause to note that the second term on the righthand side has the character of a Laplace transform, and accordingly we will use the notation 
\[
\hat{\nu}(\lambda):=\int_{[-\tau,0]} e^{\lambda s}\nu(ds), \quad \lambda\in \mathbb{C}.
\]
For measurable functions $f$ defined on $[0,\infty)$, the usual Laplace transform is defined by 
\[
\hat{f}(\lambda):=\int_0^\infty e^{-\lambda s}f(s)\,ds
\]
for $\lambda$ in appropriate regions of $\mathbb{C}$.
 
Returning to a discussion of the solutions of the characteristic equation, it is standard that the set $\Lambda:=\{\lambda \in \mathbb{C}:h(\lambda)=0\}$ is non--empty and that there is a finite $v_0(\nu)\in \mathbb{R}$ such that 
\[
v_0(\nu) = \sup\{\text{Re}(\lambda): \lambda\in \Lambda\}.
\] 
Finally, it is the case that $r(t)\to 0$ as $t\to\infty$ is equivalent to $v_0(\nu)<0$. The significance of the number $v_0(\nu)$ is that it enables us to obtain a definite exponential bound on $r$. Specifically, for all $\alpha > \emph{v}_0(\nu)$ we have $ r(t)= o(\text{exp}(\alpha t))$ for $t\to \infty$. Indeed, we have a global exponential bound: for each $\alpha>v_0(\nu)$ there is a $K=K_\alpha>0$ such that $|r(t)|\leq K_\alpha e^{\alpha t}$ for all $t\geq 0$. This global exponential bound is inherited by $x_0(\cdot,\psi)$ which can be seen via a variation of constants formula 
\[
x_0(t,\psi)=r(t)\psi(0) + \int_{[-\tau,0]}\left(\int_s^0r(t+s-u)\psi(u)du\right)\nu(ds), \quad t\geq 0.
\]
Taking the triangle inequality, one obtains the bound 
\[
|x_0(t,\psi)|\leq C_\alpha e^{\alpha t}\sup_{s\in [-\tau,0]} |\psi(s)|, \quad t\geq 0,
\]
where $C_\alpha>0$ is a constant depends on $\alpha$ and $\nu$, but not on $\psi$. Since we assume that $\psi$ is random with $\mathbb{E}[\sup_{s \in [-\tau,0]} |\psi(s)|^2]<+\infty$, we get the estimate
\[
\mathbb{E}[x_0^2(t;\psi)]\leq C_\alpha^2 e^{2\alpha t} \mathbb{E}\left[\sup_{s\in [-\tau,0]} |\psi(s)|^2\right], \quad t\geq 0,
\]
for all $\alpha>v_0(\nu)$. We therefore observe that if $v_0(\nu)<0$, then all solutions tend to zero exponentially fast for deterministic initial conditions, and also exponentially fast in mean square if the initial function has a finite mean square in the sense given above.

The converse of this result is also true (namely that if all solutions of \eqref{eq. Deterministic x0} tend to zero, then $v_0(\nu)<0$). 

To see this, we make a general observation. If $\lambda \in \Lambda$, and $\psi(t)=e^{\lambda t}$ for $t\in [-\tau,0]$, then $x_0(t,\psi)=e^{\lambda t}$ for all $t\geq 0$. If $\lambda$ is real, then this furnishes a real--valued solution; in the case that $\lambda\in \mathbb{C}$, we can use the observation that the conjugate of $\lambda$, $\bar{\lambda}$ is also in $\Lambda$, to get real valued solutions. Note first that $x(t,a\psi_1+b\psi_2)=ax(t,\psi_1)+bx(t,\psi_2)$ for $t\geq 0$ and any $a,b\in \mathbb{C}$ and any continuous complex--valued initial functions $\psi_1$ and $\psi_2$. Taking 
$\psi_1(t)=e^{\lambda t}$ and $\psi_2(t)=e^{\bar{\lambda} t}$ and $a=b=1/2$, we see that the real--valued initial function $\psi(t)=\text{Re}(e^{\lambda t})$ for $t\geq 0$ gives rise to the real--valued solution $x(t,\psi)=\text{Re}(e^{\lambda t})$ for $t\geq 0$; likewise, taking $a=1/(2i)$ and $b=-1/(2i)$, the real--valued initial function $\psi(t)=\text{Im}(e^{\lambda t})$ for $t\geq 0$ gives rise to the real--valued solution $x(t,\psi)=\text{Im}(e^{\lambda t})$ for $t\geq 0$.  

Now, let $x_0(t,\psi)\to 0$ as $t\to\infty$ for all $\psi\in C([-\tau,0],\mathbb{R})$. Suppose, by way of contradiction, that $v_0(\nu)\geq 0$. Since in fact $\sup\{\text{Re}(\lambda): \lambda\in \Lambda\} = \max\{\text{Re}(\lambda): \lambda\in \Lambda\}$ we have that there is $\lambda\in \Lambda$ such that $\text{Re}(\lambda)=v_0(\nu)$. Now take $\psi(t)=\text{Re}(e^{\lambda t})$ for $t\in [-\tau,0]$. Then $x(t,\psi)= \text{Re}(e^{\lambda t})$ for $t\geq 0$. But since $\text{Re}(\lambda)=v_0(\nu)\geq 0$, $\limsup_{t\to\infty} |x(t,\psi)|>0$. But this contradicts the supposition that 
$x(t,\psi)\to 0$ as $t\to\infty$, so we have that $v_0(\nu)< 0$, as needed.   

Here, and in what follows, we place great weight on understanding the asymptotic behaviour of components in the solution of the perturbed stochastic equation, because it is a common feature in our proofs that we decouple the behaviour of the perturbed stochastic equation into parts which depend on either the functional appearing in the diffusion coefficient or on the underlying deterministic equation, i.e., the resolvent. Before discussing the mean square asymptotic stability of \eqref{eq. Unperturbed SFDE} we need to introduce the notation
\[
 G(f_t):=\int_{[-\tau,0]}f(t+u)\mu(du), \text{ for all } f\in C[0,\infty),
\]
for a measure $\mu\in M$. We recall the main result\footnote{This result excludes the pathological case in which deterministic solutions solve the stochastic equation i.e., $G([x_0]_t) \equiv 0$ for all initial functions, meaning the solution is no longer stochastic. We also exclude such cases throughout this paper.} from Appleby et al.~\cite{AMR} which states
\begin{equation} \label{cond. unperturbed stability}
\lim_{t\to \infty}\mathbb{E}[U^2(t,\psi)] = 0 \text{ for all $\psi$ obeying \eqref{eq.second moment psi}} \quad \iff \quad 
\begin{cases} 
r \in L^2(\mathbb{R}_+),\\
\Vert G(r_{\cdot})\Vert_{L^2(\mathbb{R}_+)}<1,
\end{cases}
\end{equation}
where for an $L^2$ function $f$ on $[0,\infty)$ we use the conventional notation 
\[
\Vert f\Vert_{L^2(\mathbb{R}_+)}:=\left(\int_0^\infty f^2(s)\,ds\right)^{1/2}.
\]
Although this stability theorem is stated in terms of objects that are not part of the problem data\footnote{It should be noted that upon further analysis, these conditions can be expressed as explicit conditions on the measures 
$\mu$ and $\nu$. This however is not the aim of current paper and will be addressed in a future work.}, it still offers some support to our intuition that provided the underlying deterministic system is stable (the first condition in \eqref{cond. unperturbed stability}) and the perturbation term is ``\emph{small}" (the second condition in \eqref{cond. unperturbed stability}), then the unperturbed stochastic equation remains asymptotically stable in a mean square sense.\\
\newline
We now introduce the perturbed equation which is defined on the same probability space introduced above
\begin{align} \label{eq. Perturbed Stochastic X}
    dX(t) & =\left(f(t)+\int_{[-\tau,0]}X(t+s)\nu(ds)\right)dt + \left(g(t)+\int_{[-\tau,0]}X(t+s)\mu(ds)\right)dB(t), \quad t\geq 0,\\
    X(t) & =\psi(t), \quad t\leq 0, \nonumber
\end{align}
where $f,g \in C(\mathbb{R}_+;\mathbb{R})$ are deterministic functions, and $\psi$ has the same properties as in the solution of \eqref{eq. Unperturbed SFDE}. As with condition \eqref{eq.second moment psi}, \textit{we will assume this continuity and determinism of $f$ and $g$ throughout the paper without further reference}. For existence and uniqueness of solutions with finite second moments of \eqref{eq. Perturbed Stochastic X}, in the sense described above for \eqref{eq. Unperturbed SFDE}, we refer the reader to the monograph by Mao~\cite{Mao}. Additionally we introduce notation for the perturbed deterministic equation
\begin{align} \label{eq. Perturbed Deterministic x }
\dot{x}(t) & =f(t)+\int_{[-\tau,0]}x(t+u)\nu(du), \quad t\geq 0,\\
x(t)& = \psi(t), \quad t\leq0.   \nonumber
\end{align}

Using Laplace transforms one can readily obtain a variation of constants formula for solutions of \eqref{eq. Perturbed Deterministic x } namely
\begin{equation} \label{eq. VOC deterministic perturbed x}
    x(t,\psi)=r(t)\psi(0) + \int_{[-\tau,0]}\left(\int_s^0r(t+s-u)\psi(u)du\right)\nu(ds)+\int_0^t r(t-s)f(s)ds, 
\end{equation}
for $t \geq0$. In order to represent such solutions and other expressions efficiently, we introduce standard  notions of convolution used throughout this paper. We denote the convolution of two functions on $[0,\infty)$ by 
\[
(f \ast g)(t) \coloneqq \int_0^tf(t-s)g(s)ds, \quad t\geq 0.
\]
If $\mu$ is a finite measure on $[-\tau,0)$, and $f:[0,\infty)\to\mathbb{R}$ then 
\[
(f\ast \mu)(t):=\int_{[-\tau,0]} f(t+s)\mu(ds), \quad t\geq 0.
\]
In order to use existing results on the convolution of functions with finite measures on $[0,\infty)$, we do the following. For any subset of the real line, write $-E:=\{x\in\mathbb{R}:-x\in E\}$. If $\mu \in M([-\tau,0];\mathbb{R})$, we can construct a $\tilde{\mu}\in M([0,\infty);\mathbb{R})$ by writing 
\begin{align} \label{eq. tildemu}
&\tilde{\mu}(E)=\mu(-E), \quad  \text{for any Borel set $E\subseteq [0,\tau]$} \nonumber, \\ &\tilde{\mu}(E)=0, \quad \text{for any Borel set $E$ with $E\cap[0,\tau] =\emptyset$}. 
\end{align}
Let $f:[0,\infty)\to\mathbb{R}$, and note that 
\[
(f\ast \mu)(t)=(\tilde{\mu}\ast f)(t):=\int_{[0,t]}\tilde{\mu}(ds)f(t-s), \quad t\geq 0. 
\]
It is with this consideration, along with our notation for equation \eqref{eq. Deterministic x0} that we may rewrite equation \eqref{eq. VOC deterministic perturbed x} as,
\begin{equation} \label{eq. x in terms of x_0 and x_1}
    x(t,\psi)=x_0(t,\psi)+x_1(t), \quad t \geq0,
\end{equation}
where $x_1(t) \coloneqq (r \ast f)(t)$. We do this to exploit the fact that $x_1$ is independent of $\psi$ and non--random, while $x_0$ is independent of $f$ and random: it also allows us to make use of the results we stated and deduced for $x_0$ earlier. 
 
To keep notation clean, from this point on we will frequently omit the dependence on the initial condition, will we write $x_0(t)=x_0(t,\psi)$, $x(t)=x(t,\psi)$ and $X(t)=X(t,\psi)$ for solutions of \eqref{eq. Deterministic x0}, \eqref{eq. Perturbed Deterministic x } and \eqref{eq. Perturbed Stochastic X} respectively. 

\section{Volterra Equations for the Mean Square}
Following the spirit of Appleby at al. \cite{AMR} we define a new process
\begin{equation} \label{eq. Y}
Y(t)=g(t)+\int_{[-\tau,0]}X(t+s)\mu(ds), \quad t\geq0.
\end{equation}
This allows to to readily write down a variation of constants formula for solutions of \eqref{eq. Perturbed Stochastic X}. By Lemma 6.1 from Rei\ss\, et al. \cite{RRG}, we have  
\begin{equation} \label{eq. VOC Perturbed Stochastic X}
  X(t)= 
  \begin{dcases} 
      x(t)+\int_0^tr(t-s)Y(s)dB(s), & t\geq0,\\
      \psi(t), & t \in [-\tau,0],
  \end{dcases}
\end{equation}
where $r$ is the resolvent given by equation \eqref{eq. resolvent}. Although \eqref{eq. VOC Perturbed Stochastic X} does not give an explicit solution for $X$, it does allow us to readily write down a deterministic Volterra equation for the mean square of $X$, and also an expression relating the mean square of $Y$ to the mean square of $X$. In so doing, the question of studying the mean square of the stochastic equation is converted into one of studying the solution of certain deterministic convolution integral equations, to which the extensive--- and much more widely understood--- theory of  deterministic equations can be applied. As such, the following result can be considered \textit{the most important one in the paper}. This is because it not only forms the basis for the particular asymptotic results derived here, but acts as a springboard in the future to a very complete understanding of the mean square of solutions of perturbed SFDEs, where the perturbations $f$ and $g$ may have other interesting properties.  

\begin{theorem} \label{thm. Mean Square equations for X and Y}
Let $X$ be the solution of \eqref{eq. Perturbed Stochastic X}. Then we have for all $t \geq 0$,
\begin{equation} \label{eq. Mean Sqaure X}
    \mathbb{E}[X^2(t)]=\mathbb{E}[x^2(t)]+\int_0^tr^2(t-s)\mathbb{E}[Y^2(s)]ds,
\end{equation}
where Y, defined by \eqref{eq. Y}, obeys for all $t \geq 0$,
\begin{equation} \label{eq. Mean Square Y}
    \mathbb{E}[Y^2(t)]=\mathbb{E}\left[\left(g(t)+G(x_t)\right)^2\right]+\int_0^tG^2(r_{t-s})\mathbb{E}[Y^2(s)]ds.
\end{equation}
\end{theorem}

\begin{proof}[Proof of Theorem \ref{thm. Mean Square equations for X and Y}]
Squaring \eqref{eq. VOC Perturbed Stochastic X} gives
\begin{align} \label{eq: X2}
    X^2(t)=x^2(t)+2x(t)\int_0^tr(t-s)Y(s)dB(s)+\left(\int_0^tr(t-s)Y(s)dB(s)\right)^2.
\end{align}
We first consider the second cross term on the right hand side; by letting $t_0 \in [0,t]$ be arbitrary, we define
\[
M(t) \coloneqq 2x(t_0)\int_0^t r(t_0-s)Y(s)dB(s), \quad t\geq 0.
\]
Taking expectations and using It\^o's isometry one can show
\[
\mathbb{E}[|M(t)|] \leq \mathbb{E}[x^2(t_0)] +\int_0^t r^2(t_0-s)\mathbb{E}[Y^2(s)]ds,
\]
so $\mathbb{E}[|M(t)|] < \infty$ for all $t \geq 0$. On the other hand, because $x(t_0)$ has finite expectation, is independent of $B$ and is $\mathcal{F}(0)$--measurable, we have that for $t\geq s\geq 0$ 
\[
\mathbb{E}[M(t)|\mathcal{F}(s)]=2x(t_0)\mathbb{E}\left[  \int_0^t r(t_0-u)Y(u)dB(u)\Big|\mathcal{F}(s)\right] = M(s),
\]
using the fact that the second factor in $M$ is a martingale. Therefore, $M$ is a martingale and so $\mathbb{E}[M(t)]=0$ for all $t\geq t_0\geq 0$. In particular $\mathbb{E}[M(t_0)]=0$. As $t_0$ was chosen arbitrarily, this means that
\[
\mathbb{E}\left[2x(t)\int_0^tr(t-s)Y(s)dB(s)\right]=0, \quad t\geq 0.
\]
To deal with the squared term in \eqref{eq: X2}, we proceed similarly. Fix $t_0\geq 0$ and for $t\geq 0$ define  
\[
N(t)=\int_0^t r(t_0-s)Y(s)dB(s), \quad t\geq 0.
\]
Since $X$ has finite second moments, so does $Y$, and therefore $N$ is a martingale with finite second moments. Therefore, by It\^o's isometry, we have 
\[
\mathbb{E}[N^2(t)] = \int_0^t r^2(t_0-s)\mathbb{E}[Y^2(s)]ds.
\]
Now take $t=t_0$, so that 
\[
\mathbb{E}\left[\left(\int_0^{t_0} r(t_0-s)Y(s)dB(s)\right)^2\right] = \int_0^{t_0} r^2(t_0-s)\mathbb{E}[Y^2(s)]ds.
\]
Since $t_0$ is arbitrary, we may replace it  by $t$, and therefore taking expectations on both sides of \eqref{eq: X2}, we get 
\[
\mathbb{E}[X^2(t)]=\mathbb{E}[x^2(t)]+\int_0^tr^2(t-s)\mathbb{E}[Y^2(s)]ds,
\]
as required. Next we prove that $Y(t)$ obeys \eqref{eq. Mean Square Y}. Letting $t \geq \tau$, and using Fubini's theorem for stochastic integrals we can show
\begin{align*}
    Y(t) & = g(t) + \int_{[-\tau,0]}X(t+s)\mu(ds)\\
    & = g(t) + \int_{[-\tau,0]}\left( x(t+s)+\int_0^{t+s}r(t+s-u)Y(u)dB(u)\right)\mu(ds)\\
    & = g(t) + G(x_t) +  \int_{[-\tau,0]}\int_0^{t+s}r(t+s-u)Y(u)dB(u)\mu(ds)\\
    & = g(t) + G(x_t) + \int_0^t\left(\int_{[\max\{-\tau,u-t\},0]}r(t+s-u)\mu(ds)\right)Y(u)dB(u)\\
    & = g(t) + G(x_t) + \int_0^tG(r_{t-u})Y(u)dB(u),
\end{align*}
where in the last line we used the fact that $r(t) = 0$ for all $t < 0$. Notice that the integral is of the same form as that in \eqref{eq. VOC Perturbed Stochastic X}, and that the first term on the righthand side has the same properties as $x$, namely, independence from $B$, finite moments and $\mathcal{F}(0)$--measurability. Therefore, we can compute $\mathbb{E}[Y^2(t)]$ by following exactly the same steps as used to compute $\mathbb{E}[X^2(t)]$ above. Doing this, we get 
\[
\mathbb{E}[Y^2(t)]=\mathbb{E}\left[\left(g(t)+G(x_t)\right)^2\right]+\int_0^tG^2(r_{t-s})\mathbb{E}[Y^2(s)]ds,
\]
for $t>\tau$. It remains to obtain the corresponding integral equation for $\mathbb{E}[Y^2(t)]$ for $t\in [0,\tau]$. Proceeding as before, we have that
\begin{align*}
    Y(t) & = g(t) + \int_{[-\tau,-t]}X(t+s)\mu(ds)+ \int_{[-t,0]}X(t+s)\mu(ds)\\
     & = g(t) + \int_{[-\tau,-t]}\psi(t+s)\mu(ds)+ \int_{[-t,0]}x(t+s)\mu(ds)\\
     & \qquad+  \int_{[-t,0]}\int_0^{t+s}r(t+s-u)Y(u)dB(u)\mu(ds)\\ 
     & = g(t) + G(x_{t}) +  \int_{[-t,0]}\int_0^{t+s}r(t+s-u)Y(u)dB(u)\mu(ds).
\end{align*}
Then by invoking the stochastic Fubini theorem, and once again using the fact that $r(t)=0$ for all $t\leq 0$, we arrive once more at 
\[
Y(t) = g(t) + G(x_t) + \int_0^tG(r_{t-u})Y(u)dB(u), \quad t\in [0,\tau].
\]
Squaring and taking expectations, we get the same expression for $\mathbb{E}[Y^2(t)]$ deduced above on $[\tau,\infty)$, and this completes the proof.
\end{proof}

\section{Main Results for Asymptotic Stability}
With the integral equations for $\mathbb{E}[Y^2]$ and $\mathbb{E}[X^2]$ in hand, we are now ready to present asymptotic results: in this section we present three results characterising certain types of stability for solutions of \eqref{eq. Mean Sqaure X}. Although the equations for the mean square are deterministic, there are two special challenges to meet. Firstly, these integral equations are written in terms of objects such as $x$, which are not part of the problem data, and our goal is to determine conditions for asymptotic behaviour which can be stated more directly in terms of, and with minimal dependence on, the problem data. Secondly,  we wish to present necessary and sufficient conditions for certain types of stability, and we will try to do this by imposing conditions on the perturbing functions $f$ and $g$ \textit{which do not depend on the resolvent $r$, or the measures $\nu$ and $\mu$}. 
\subsection{Reformulation and preliminaries} 
In trying to keep notation as clean as possible we find the following definitions to be useful when proving all results in this section. If we define $Z(t) \coloneqq (r^2 \ast \mathbb{E}[Y^2])(t)$, then the equations for the mean square become
\begin{align} \label{eq. Z}
    \mathbb{E}[X^2(t)] & = \mathbb{E}[x^2(t)]+Z(t), \nonumber \\ 
    Z(t) & = \left(r^2 \ast \mathbb{E}\left[(g+G(x_\cdot))^2\right]\right)(t)+(G^2(r_\cdot) \ast Z)(t), 
\end{align}
for $t \geq 0$; the second equation was obtained by taking the convolution with $r^2$ across equation \eqref{eq. Mean Square Y}. Further defining
\begin{equation} \label{eq. gamma}
    \gamma(t)\coloneqq \left(r^2 \ast \mathbb{E}\left[(g+G(x_\cdot))^2\right]\right)(t), \quad t\geq 0,
\end{equation}
finally yields
\begin{equation} \label{eq. Z in terms of gamma}
Z(t)  = \gamma(t)+(G^2(r_\cdot) \ast Z)(t), \quad t\geq 0.
\end{equation}
Since the behaviour of $x$ depends on that of $x_0$ and $x_1=r\ast f$, which are known directly, the asymptotic behaviour of the mean square of $X$ is clinched by getting the asymptotic behaviour of $Z$. In this direction, it makes sense to introduce an integral resolvent $\rho$ which is independent of $\gamma$, but in terms of which $Z$ can be expressed. Let 
$\rho$ obey the equation
\begin{equation} \label{eq. rho}
  \rho(t) = G^2(r_t) + (G^2(r_\cdot) \ast \rho)(t), \quad t \geq 0,
\end{equation}
(see \cite[Ch. 2]{GLS}). Then
\[
Z(t)=\gamma(t)+(\rho\ast \gamma)(t), \quad t\geq 0.
\]
We begin this section with a lemma that provides an integrability result on $\rho$. 

To do so, we need first to deal with a special case, in which  $\|G(r)\|_{L^2(\mathbb{R}^+)}=0$. If this is the case, then $G(r_t)=0$ a.e. $t\geq 0$. Taking Laplace transforms across this equation gives $\hat{\mu}(z)\hat{r}(z)=0$ for values of $z\in \mathbb{C}$ for which $\text{Re}(z)>v_0(\nu)$ (for these values of $z$ we are guaranteed that $\hat{r}(z)$ is well--defined; since $\mu$ is finite and supported on $[-\tau,0]$, $\hat{\mu}(z)$ is defined for all $z\in \mathbb{C}$). But since for $\text{Re}(z)>v_0(\nu)$, we have $z-\hat{\nu}(z)\neq 0$ and $\hat{r}(z)(z-\hat{\nu}(z))=1$, it follows that $\hat{\mu}(z)=0$  for all $\text{Re}(z)>v_0(\nu)$. This implies that $\mu(E)=0$ for all Borel sets $E\subseteq [-\tau,0]$. As a consequence, in the case when $\|G(r)\|_{L^2(\mathbb{R}^+)}=0$, we have that  $X$ obeys the SFDE 
\[
dX(t)=\left(f(t)+\int_{[-\tau,0]}X(t+s)\nu(ds)\right)\,dt + g(t)\,dB(t), \quad t\geq 0.
\]
Thus $X$ has the representation  
\[
X(t,\psi) = x_0(t,\psi)+\int_0^t r(t-s)f(s)\,ds + \int_0^t r(t-s)g(s)\,dB(s), \quad t\geq 0,
\]
or 
\[
X(t,\psi) = x(t,\psi) + \int_0^t r(t-s)g(s)\,dB(s), \quad t\geq 0, \quad \text{a.s.}
\]
The mean square is given explicitly by 
\begin{equation} \label{eq. XsqG0}
\mathbb{E}[X^2(t)] = \mathbb{E}[x^2(t)] + \int_0^t r^2(t-s)g^2(s)ds,
\end{equation}
and can be studied by direct deterministic methods. To summarise, we have shown that $\|G(r_\cdot)\|_{L^2(\mathbb{R}^+)}=0$ if and only if $\mu$ is almost everywhere zero (the reverse implication is trivial) and in these cases, the mean square is given directly by \eqref{eq. XsqG0}. We will sometimes need to treat the situation when $\mu$ is zero idiosyncratically in our proofs, and in many cases appealing to \eqref{eq. XsqG0} directly suffices. 

The following lemma is needed when  $\|G(r_\cdot)\|_{L^2(\mathbb{R}^+)}>0$; in the case when  $\|G(r_\cdot)\|_{L^2(\mathbb{R}^+)}=0$, the lemma is not needed, and a direct appeal to \eqref{eq. XsqG0} can be made instead.
\begin{lemma} \label{lem. exponential integrability of rho}
    Let $\rho$ be the integral resolvent of \eqref{eq. Z in terms of gamma}, $r \in L^2(\mathbb{R}_+)$ and $0<\Vert G(r_{\cdot})\Vert_{L^2(\mathbb{R}_+)}^2<1$. Then there is an $\alpha>0, \alpha'>0$ such that the function
    \begin{equation} \label{eq. Capital Gamma function}
        \Gamma(\lambda) \coloneqq \int_0^\infty e^{2\lambda s}G^2(r_s)ds,
    \end{equation}
    is well defined for $\lambda \in [0,\alpha)$, and furthermore,
    \begin{equation*}
        \int_0^\infty e^{2\epsilon s}\rho(s)ds < \infty,
    \end{equation*}
    for all $\epsilon \in [0,\alpha')$ where $\alpha'$ is the unique number such that $\Gamma(\alpha')=1$ . 
\end{lemma}
\begin{proof}[Proof of Lemma \ref{lem. exponential integrability of rho}]
    To prove the first assertion, note the assumption on $r\in L^2(\mathbb{R}_+)$ gives us the estimate
    \[
    |r(t)| \leq Ke^{-\alpha t}, \quad t \geq 0,
    \]
    and for some $K>0$ and $\alpha>0$ (where $-\alpha>v_0(\nu)$). One can extend this estimate to $G(r_\cdot)$: 
    \begin{align*}
        |G(r_t)| & = \left| \int_{[-\tau,0]}r(t-s)\nu(ds)\right| \leq e^{-\alpha t}\int_{[-\tau,0]}e^{\alpha s}|\nu|(ds)
        =K'e^{-\alpha t},
    \end{align*}
    for some $K'>0$, thus we have $G^2(r_t) \leq Ce^{-2\alpha t}$. Here we are using the conventional notation $|\nu|$ for the total variation measure of $\nu$, which is a positive and finite measure in $M[-\tau,0]$ (see e.g., \cite[Thm. 6.2]{Rudin2}). We also exploit here the fundamental estimate 
    \[
   \left| \int_{[-\tau,0]} f(t+s)\nu(ds)\right| \leq \int_{[-\tau,0]} |f(t+s)||\nu|(ds),
    \] 
    for measurable functions $f$, which can be deduced from the case for finite measures on $[0,\infty)$ (see \cite[Thm. 3.4.5]{GLS}).  
    
    Using the estimate for $G(r_\cdot)$ in \eqref{eq. Capital Gamma function} 
		gives, for all $0 \leq \lambda<\alpha$,
    \[
    \Gamma(\lambda) \leq \int_0^\infty e^{-2s(\alpha-\lambda)}ds < +\infty.
    \]
		On the other hand, $\Gamma$ is clearly non--decreasing on its maximal interval of existence. Moreover, since $t\mapsto G(r_t)$ is non--trivial, it follows that either there is a finite $\beta>0$ such that $\Gamma$ is well defined on $[0,\beta)$ and $\Gamma(\lambda)\to\infty$ as $\lambda\to\beta^-$ or that $\Gamma$ is defined on $[0,\infty)$ and $\Gamma(\lambda)\to\infty$ as $\lambda\to \infty$. Therefore, $\Gamma$ is continuous and increasing, $\Gamma(0)<1$ and $\Gamma(\lambda) \to \infty$ as $\lambda \to \beta^-$ (where $\beta=\infty$ is possible). Thus by the intermediate value theorem, there is a unique $\alpha'<\beta$ such that $\Gamma(\alpha')=1$. Since $\Gamma$ is well defined on $[0,\alpha)\subseteq [0,\beta)$ we have $\alpha\leq \beta$. Thus, we may choose any $\epsilon\in (0,\alpha')$ such that $\Gamma(\epsilon)<1$.
	
	For the second statement, choose such an $\epsilon \in (0,\alpha')$, and scale equation \eqref{eq. rho} by $e^{2\epsilon t}$ to get
    \[
      \rho(t)e^{2\epsilon t} = G^2(r_t)e^{2\epsilon t} + \int_0^\infty e^{2\epsilon(t-s)}G^2(r_{t-s})\cdot e^{2\epsilon s}\rho(s)ds, \quad t \geq 0.
    \]
    Using the notation $\rho_\epsilon(t) \coloneqq \rho(t)e^{2\epsilon t}$ and $G_\epsilon^2(r_t) \coloneqq e^{2\epsilon t} G^2(r_{t})$ we can rewrite the above equation as
    \begin{equation} \label{eq. scaled rho}
        \rho_\epsilon(t) = G_\epsilon^2(r_t) + \int_0^\infty G_\epsilon^2(r_{t-s})\rho_\epsilon(s)ds, \quad t\geq 0.
    \end{equation}
    Now integrating equation \eqref{eq. scaled rho}, applying Fubini's theorem and using the fact that $r(t)=0$ for $t<0$, we see that
\[
\int_0^{\infty}\rho_\epsilon(s)ds = \dfrac{\int_0^{\infty}G_\epsilon^2(r_s)ds}{1-\int_0^{\infty}G_\epsilon^2(r_s)ds},
\]
where we have exploited the fact that $\Gamma(\epsilon)=\int_0^{\infty}G_\epsilon^2(r_s)ds<1$. Thus
\[
\int_0^{\infty}\rho_\epsilon(s)ds=\int_0^{\infty}e^{2\epsilon s}\rho(s)ds,
\]
is well defined for all $\epsilon \in [0,\alpha')$.
\end{proof}
\subsection{Statement of main results}
With the above preliminaries dispensed with, we are in a position to state our main results on asymptotic behaviour. We consider three types of convergence of solutions, and describe necessary and sufficient conditions on the perturbations $f$ and $g$, as well as the underlying unperturbed equation $U$, for which each type of convergence result holds. 

The common theme of the results is two--fold: first of all, in order that solutions $X$ of the perturbed equation have the appropriate behaviour in mean square, it is necessary that the solutions $U$ of the unperturbed equation tend to zero in mean square, and indeed this forms part of the sufficient conditions for convergence of the perturbed equations too. The second common feature is that the behaviour of the perturbing terms $f$ and $g$ can be quite irregular or ``out of control'' on a pointwise basis, but nevertheless the mean square of the solution will be well--behaved. Roughly speaking, if $g$ is such that $\int_t^{t+1} g^2(s)\,ds$ has the appropriate decay property, and $\int_{t}^{t+\delta} f(s)\,ds$ has the appropriate decay property for all $\delta\in (0,1]$, then $\mathbb{E}[X^2]$ will have the decay property. In fact, these average conditions on $f$ and $g$ turn out to be necessary for the appropriate decay property in the mean square of $X$: such decay in the mean square implies these ``sectional averages'' of $f$ and $g$ must have the stipulated decay too.     

\begin{theorem} \label{thm. mean square X to zero}
Let $X$ be the solution to equation \eqref{eq. Perturbed Stochastic X}. Suppose that $\psi$ obeys \eqref{eq.second moment psi}. Then the following conditions (\textbf{A}) and (\textbf{B}) are equivalent:
\begin{itemize}
    \item [(\textbf{A})]
        \begin{itemize}
            \item[(i)] $r \in L^2(\mathbb{R}_+)$,
            \item[(ii)] $\Vert G(r_{\cdot})\Vert_{L^2(\mathbb{R}_+)}<1$,
            \item [(iii)] For all $\delta \in (0,1],$ $\int_t^{t+\delta}f(s)ds \to 0$ as $t \to \infty$,
            \item[(iv)] $\int_{t}^{t+1}g^2(s)ds \to 0$ as $t \to \infty$.
        \end{itemize}
    \item[(\textbf{B})] $\lim_{t \to \infty}\mathbb{E}[X^2(t,\psi)]=0$ for all $\psi \in C([-\tau,0];\mathbb{R})$.
\end{itemize}    
\end{theorem}

Note that conditions (i) and (ii) in \textbf{(A)} are equivalent to $\mathbb{E}[U^2(t,\psi)]\to 0$ as  $t\to\infty$ for all $\psi$. Therefore, the convergence of the solution in mean square to zero is equivalent to the global asymptotic stability of the unperturbed equation, coupled with the decay properties of $f$ and $g$ in (iii) and (iv). 

We note that the conditions (iii) and (iv) are fulfilled for functions $f$ and $g$ for which $f(t)\to 0$ as $t\to\infty$ and $g(t)\to 0$ as $t\to\infty$. However, $f$ and $g$ can be substantially less well--behaved, and still the conditions (iii) and (iv) can be fulfilled. 

For instance, let $n\in \mathbb{N}$, and $a_n<1/2$ and $w_n:=1/2-a_n$ and $h_n$ be positive sequences. Suppose that $f(t)=0$ for $[n,n+a_n]$ and $[n+1-a_n, n+1]$ and on $[n+a_n,n+a_n+w_n]$, $f$ is linear with $f(n+a_n)=0$ 
and $f(n+a_n+w_n)=h_n$, while on $[n+a_n+w_n,n+1-a_n]$, $f$ is linear with $f(n+1-a_n)=0$. Then $f$ is continuous, has a spike of width $2w_n$ and maximal height $h_n$, and obeys 
\[
\int_n^{n+1} f(t)\,dt = w_n h_n.
\]  
Since $f$ is non--negative, we have that condition (iii) is satisfied if and only if $w_n h_n \to 0$ as $n\to\infty$. On the other hand, if $h_n\to \infty$ as $n\to\infty$, $f$ can be unbounded, with the running maximum of $f$ having an arbitrarily fast rate of growth. Thus the spike in $f$ can be arbitrarily high, provided it has a sufficiently short duration (i.e. $w_n=o(1/h_n)$ as $n\to\infty$, and $o$ is the conventional ``little $o$'' Landau notation). 

Taking $g=\sqrt{f}$ in the above supplies an example with $\limsup_{t\to\infty} |g(t)|=+\infty$ for which condition (iv) applies. 

An example of an $f$ with no sign restrictions, which satisfies (iii), but which has very bad pointwise behaviour, is 
\begin{equation} \label{eq. exphfsin}
f(t)=e^{\alpha t} \sin(e^{\beta t}), \quad t\geq 0,
\end{equation}
where $0<\alpha<\beta$. Write $T=e^{\beta t}$, $A=e^{\beta \delta}$. Note that $\epsilon:=1-\alpha/\beta\in (0,1)$. Then
\[
\int_t^{t+\delta} f(s)\,ds = \frac{1}{\beta}\int_T^{AT} u^{-\epsilon} \sin(u)\,du.
\] 
Integrating the right hand side by parts we see that it is $O(T^{-\epsilon})$ 
as $T\to\infty$, where we use the conventional ``big $O$'' Landau notation. Therefore 
\[
\int_t^{t+\delta} f(s)\,ds = O(e^{-(\beta-\alpha)t}), \quad t\to\infty,
\]
for each $\delta>0$. 

We notice that the condition with an absolute value inside the integral is too restrictive.  For instance, if $f$ obeys 
\[
\int_t^{t+\delta} |f(s)|\,ds\to 0 \quad \text{as $t\to\infty$, for any $\delta>0$},
\] 
then this condition implies \textbf{(A)}(iii), and is equivalent to \textbf{(A)}(iii) when $f$ does not change sign on $[0,\infty)$. However, it can be shown for $f$ obeying \eqref{eq. exphfsin} with $\alpha\in (0,\beta)$ the integral $\int_t^{t+\delta} |f(s)|\,ds$ diverges as $t\to\infty$. To see this, use the notation above; integration by substitution gives  
\[
\int_{t}^{t+\delta} |f(s)|\,ds = \frac{1}{\beta}\int_T^{AT} u^{-\epsilon} |\sin(u)|\,du.
\] 
Recall that $\epsilon\in (0,1)$ and consider intervals on which $|\sin(u)|\geq 1/2$. Bounding the integral below by considering only these intervals, we see that the lower bound grows at a rate 
\[
C \int_T^{AT} u^{-\epsilon}\,du \geq C'T^{1-\epsilon}, 
\]
where $C$ and $C'$ are $T$--independent and strictly positive. Since $\epsilon\in (0,1)$, the lower bound diverges, and hence 
\[
\int_{t}^{t+\delta} |f(s)|\,ds \geq C'e^{(\beta-\alpha)t}, \quad t\to\infty. 
\]

In applications, understanding when convergence to limiting values is exponentially fast is often important. We turn to this next. First, it is not hard to show that the mean square asymptotic stability of the unperturbed equation implies the exponential decay in the mean square of $U$. It is therefore natural to ask what conditions on $f$ and $g$ preserve this exponential convergence in the mean square of solutions of \eqref{eq. Perturbed Stochastic X}. Once again, this exponential convergence occurs if and only if the unperturbed equation is mean square asymptotically stable, and $f$ and $g$ obey an exponential decay bound. 

\begin{theorem} \label{thm. mean square X exponential decay}
	Let $X$ be the solution to equation \eqref{eq. Perturbed Stochastic X}. Suppose that $\psi$ obeys \eqref{eq.second moment psi}. Then the following conditions (\textbf{A}) and (\textbf{B}) are equivalent:
	\begin{itemize}
		\item [(\textbf{A})]
		\begin{itemize}
			\item[(i)]  $r \in L^2(\mathbb{R}_+)$,
			\item[(ii)]  $\Vert G(r_{\cdot})\Vert_{L^2(\mathbb{R}_+)}<1$,
			\item [(iii)]  There is a $\beta_1>0$ such that $\int_0^\infty e^{2\beta_1 s}g^2(s)ds < \infty$. 
			\item[(iv)]  There is a $\beta_2>0$ such that $t \mapsto \left| \int_0^t e^{\beta_2 s}f(s)ds\right|$ is uniformly bounded. 
		\end{itemize}
		\item[(\textbf{B})] $\mathbb{E}[X^2(t;\psi)] \leq C^2(\psi,f,g)e^{-2\alpha(f,g) t}$, for all $\psi \in C([-\tau,0];\mathbb{R})$ with $\alpha(f,g)>0$.
	\end{itemize}  
\end{theorem}
We notice once again from conditions (iii) and (iv) that neither $f$ nor $g$ need to obey pointwise exponential bounds, but that rather they exhibit exponential decay ``on average''. 

The conditions \textbf{(A)}(iii) and \textbf{(A)}(iv) are equivalent to conditions which appear stronger, and give more freedom to choose the exponents $\beta_1$ and $\beta_2$. In fact, \textbf{(A)}(iii) and \textbf{(A)}(iv) give exponential integrability for all $\beta$ sufficiently small. This is clear in the condition for $g$: if
\[
\int_0^\infty e^{2\beta_1 s}g^2(s)ds < \infty,
\]  
then obviously for all $\beta<\beta_1$,
\[
\int_0^\infty e^{2\beta s}g^2(s)ds \leq \int_0^\infty e^{2\beta_1 s}g^2(s)ds < \infty.
\]
However, this is perhaps less obvious in the case of the condition \textbf{(A)}(iii) on $f$. Assume, as in \textbf{(A)}(iii), that there is a $\beta_2>0$ and $B>0$ such that 
\[
\left|\int_0^t e^{\beta_2 s}f(s)ds\right|\leq B, \quad t\geq 0. 
\]
We will show that this implies 
\[
\left|\int_0^t e^{\beta s}f(s)ds\right|\leq 2B, \quad t\geq 0.
\]
for all $\beta\in (0,\beta_2]$. 

To prove this claim, let $\beta\in (0,\beta_2]$ and define $u'_\beta(t)=-\beta u_\beta(t)+f(t)$ for $t\geq 0$, with $u_\beta(0)=0$. Notice that $|u_{\beta_2}(t)|\leq Be^{-\beta_2 t}$ for all $t\geq 0$ by hypothesis. Let $\delta_\beta=u_\beta-u_{\beta_2}$. 
Then 
\[
\delta_\beta'=-\beta u_\beta+\beta_2 u_{\beta_2} = -\beta (\delta_\beta+u_{\beta_2})+\beta_2 u_{\beta_2}.
\]
Therefore
\[
\delta_\beta(t)= \int_0^t e^{-\beta(t-s)}(\beta_2-\beta)u_{\beta_2}(s)\,ds, \quad t\geq 0.
\]
Thus
\[
|\delta_\beta(t)|\leq B(\beta_2-\beta) e^{-\beta t} \int_0^t e^{(\beta-\beta_2) s} \,ds
\leq B(\beta_2-\beta) e^{-\beta t} \int_0^\infty e^{-(\beta_2-\beta) s} \,ds
=Be^{-\beta t}.
\]
Therefore $u_\beta(t)=\delta_\beta(t)+u_{\beta_2}(t)$ obeys 
\[
|u_\beta(t)|\leq Be^{-\beta t} + B e^{-\beta_2 t}\leq 2Be^{-\beta t}.
\] 
But since $u_\beta$ is the convolution of $e^{-\beta t}$ and $f$, this gives 
\[
\left|\int_0^t e^{-\beta(t-s)}f(s)\,ds\right| \leq 2B e^{-\beta t}, \quad t\geq 0,
\]
which gives the desired $\beta$--uniform estimate claimed above, namely
\[
\left|\int_0^t e^{\beta s}f(s)\,ds\right| \leq 2B, \quad t\geq 0, \quad \beta\in (0,\beta_2].
\]

Notice the character of the exponential bound in \textbf{(B)}: the estimate of the rate of decay $\alpha$ can depend on $f$ and $g$ (and of course, on $r$), but it does not depend on $\psi$: the $\psi$--dependence is instead confined to the multiplier of the decaying exponential. Of course, we also expect $f$--, $g$-- and $r$--dependence in this multiplier. In the proof, we do not attempt to make a very fine estimate of $\alpha$: however, scrutiny of the proof suggests that the faster the decay in $r$, $g$ and $f$, larger is the estimate on $\alpha$, and the faster the rate of mean square convergence to zero.  

As pointed out earlier, the function $f(t)=e^{\alpha t}\sin(e^{\beta t})$ for $t\geq 0$, where $0<\alpha<\beta$,  obey an exponentially decaying estimate of the form
\[
\int_{t}^{t+\delta} f(s)\,ds = \text{O}(e^{-(\beta-\alpha)t}), \quad t\to\infty
\]   
for any choice of $\delta>0$, despite the fact that $f$ itself is exponentially unbounded. We show now that this exponential decay arises in exactly the form necessary for Theorem \ref{thm. mean square X exponential decay}. To see this, note for any $\eta>0$ that  
\[
\int_0^t e^{\eta s} f(s)\,ds = \int_0^t e^{(\eta+\alpha)s} \sin(e^{\beta s})\,ds
= \frac{1}{\beta}\int_{1}^T  u^{(\eta+\alpha)/\beta-1}\sin(u)\,du 
\]
where $T=e^{\beta t}$. Integration by parts yields
\begin{multline*}
\int_{1}^T  u^{(\eta+\alpha)/\beta-1}\sin(u)\,du
=  -T^{(\eta+\alpha)/\beta-1} \cos(T) + \cos(1) \\ + \left(\frac{\eta+\alpha}{\beta}-1\right)\int_1^T  u^{(\eta+\alpha)/\beta-2}\cos(u)\,du.
\end{multline*}
Therefore, as $T\to\infty$, the righthand side is bounded provided $0<\eta<\beta-\alpha$, and an $\eta$--independent upper bound can be obtained. Hence, for each $\eta\in (0,\beta-\alpha)$ we have that 
\[
\left|  \int_0^t e^{\eta s} f(s)\,ds \right| \leq B, \quad \text{ for all $t\geq 0$},
\]
so if $\int_0^\infty e^{2\beta_1 s}g^2(s)\,ds <+\infty$ for some $\beta_1>0$, then we will have exponential decay in the mean--square (contingent on the unperturbed equation being globally asymptotically stable in the mean--square).   

Theorem \ref{thm. mean square X exponential decay} shows that if the decay in $f$ and $g$ is not exponential, we do see exponential convergence in the mean square. However, in applications it is often of interest to know if solutions are integrable in the mean square. Thus, we ask what  conditions are necessary and sufficient for 
\[
\int_0^\infty \mathbb{E}[X^2(t,\psi)]\,dt <+\infty.
\]
As in previous theorems, the asymptotic mean square stability of the unperturbed equation is essential. But since this implies also the exponential decay to zero of the mean square of $U$, this means that  
\[
\int_0^\infty \mathbb{E}[U^2(t,\psi)]\,dt <+\infty,
\]
is necessary for the mean square integrability of $X$. Moreover, this mean square integrability will be preserved provided $f$ and $g$ satisfy the appropriate square integrability conditions.  
\begin{theorem} \label{thm. mean square X in L1}
		Let $X$ be the solution to equation \eqref{eq. Perturbed Stochastic X}. Suppose that $\psi$ obeys \eqref{eq.second moment psi}.
	Then the following conditions \textbf{(A)} and \textbf{(B)} are equivalent:
	\begin{itemize}
		\item [(\textbf{A})]
		\begin{itemize}
			\item[(i)]  $r \in L^2(\mathbb{R}_+)$,
			\item[(ii)]  $\Vert G(r_{\cdot})\Vert_{L^2(\mathbb{R}_+)}<1$,
			\item[(iii)]  $t\mapsto\int_0^t e^{-(t-s)}f(s)ds \in L^2(\mathbb{R}_+)$,          
			\item [(iv)]  $g \in L^2(\mathbb{R}_+)$.
		\end{itemize}
		\item[(\textbf{B})] $\mathbb{E}[X^2(\cdot;\psi)] \in L^1(\mathbb{R}_+)$, for all $\psi \in C([-\tau,0];\mathbb{R})$.
	\end{itemize}  
\end{theorem}
Notice that $g\in L^2(\mathbb{R}_+)$ is equivalent to 
\[t\mapsto
\int_t^{t+1} g^2(s)\,ds \in  L^1(\mathbb{R}_+) 
\]
so the condition on $g$ can still be framed in terms of the average over unit intervals, as in earlier theorems. 

By contrast the condition on $f$ in \textbf{(A)}(iii) differs from those in the previous theorems. Rather than asking that (for example)
\begin{equation}\label{eq.fdeltainL2}
t\mapsto 
\int_t^{t+\delta} f(s)\,ds \in L^2(\mathbb{R}_+) \quad \text{ for all $\delta\in (0,1]$},
\end{equation}
we ask that  $t\mapsto\int_0^t e^{-(t-s)}f(s)ds \in L^2(\mathbb{R}_+)$. However, we know from Lemma~\ref{lem. perturbed ODE} below that $\int_0^t e^{-(t-s)}f(s)ds$ tends to zero as $t\to\infty$ if and only if $\int_t^{t+\delta} f(s)\,ds\to 0$ for all $\delta\in (0,1]$. This leads us to speculate that in fact condition \textbf{(A)}(iii) is equivalent to \eqref{eq.fdeltainL2}. Indeed, showing this would be of interest for deterministic functional differential equations also, since it would enable one to show that the solution of 
\[
x'(t)=\int_{[-\tau,0]}x(t+s)\nu(ds)+f(t), \quad t\geq 0 
\] 
is in $L^2(\mathbb{R}_+)$ if and only if $t\mapsto\int_t^{t+\delta} f(s)\,ds \in L^2(\mathbb{R}_+)$ for all $\delta\in (0,1)$ (contingent on $r\in L^1(\mathbb{R}_+)$). We hope to address this conjecture in a forthcoming work.  

\section{Proofs}
We start with a result which connects the convergence condition on $\int_t^{t+\delta} f(s)\,ds$ with the convergence of a certain ordinary differential equation.
\begin{lemma} \label{lem. perturbed ODE}
Let $f\in C(0,\infty)$ and $u$ be the solution of
\begin{equation} \label{eq. perturbed ODE}
    u'(t)=-u(t)+f(t), \quad t\geq 0,
\end{equation}
with initial condition $u(0)=0$. If $\int_t^{t+\delta}f(s)ds\to 0$ as $t \to \infty$ for all $\delta \in (0,1]$, then,
\[
\lim_{t \to \infty}u(t)=0.
\]
Conversely, if $u(t)\to 0$ as $t\to\infty$, then $\int_t^{t+\delta}f(s)ds\to 0$ as $t \to \infty$ for all $\delta \in (0,1]$.
\end{lemma}

\begin{proof}[Proof of Lemma \ref{lem. perturbed ODE}]
Our assumption on $f$ gives rise to a decomposition such that we can write $f=f_1+f_2$ where $f_1 \in BC_0(\mathbb{R}_+;\mathbb{R})$, $f_2 \in L_{loc}^1(\mathbb{R}_+;\mathbb{R})$ and $\int_0^tf_2(s)ds \to 0$ as $t \to \infty$; see \cite[Lem. 15.9.2]{GLS} (recall that $BC_0(\mathbb{R}_+;\mathbb{R})$ is the space of all continuous functions from $\mathbb{R}^+$ to $\mathbb{R}$ which have a zero limit at infinity, and $L^1_{loc}(\mathbb{R}_+;\mathbb{R})$ is the space of all locally integrable functions from  from $\mathbb{R}^+$ to $\mathbb{R}$). Thus we can write
\begin{align*}
    u(t) & = \int_0^t e^{-(t-s)}f(s)ds\\
    & = \int_0^t e^{-(t-s)}f_1(s)ds + \int_0^t e^{-(t-s)}f_2(s)ds\\
    & = \int_0^t e^{-(t-s)}f_1(s)ds + \int_0^tf_2(s)ds + \int_0^t e^{-(t-s)}\left(\int_0^sf_2(u)du\right)ds,
\end{align*}
where the last line follows from integration by parts. As the convolution of an $L^1(\mathbb{R}_+)$ and a $BC_0(\mathbb{R}_+)$ function tends to zero as $t \to \infty$ (see \cite[Thm. 2.2.2]{GLS}) the above decomposition of $f$ ensures all three terms on the righthand side tend to zero as $t\to\infty$. 

We notice moreover that if $u(t)\to 0$ as $t\to\infty$, then $\int_t^{t+\delta} f(s)\,ds \to 0$ as $t\to\infty$ for any $\delta>0$. This is easily established by first integrating \eqref{eq. perturbed ODE} over $[t,t+\delta]$ and rearranging:
\[
\int_t^{t+\delta} f(s)\,ds=
u(t+\delta)-u(t)+\int_t^{t+\delta} u(s)\,ds.
\]
Now taking limits as $t\to\infty$, and using the fact that $u(t)\to 0$ as $t\to\infty$, we get the desired conclusion. 
\end{proof}

It is easy to adapt the proof to deal with the equation 
\[
u'(t)=-\beta u(t)+f(t), \quad t>0;\quad u(0)=0,
\]   
where $\beta>0$. Following the calculations above, we see that $u(t)\to 0$ as $t\to\infty$ if and only if $\int_t^{t+\delta} f(s)\,ds \to 0$ as $t\to\infty$ for all $\delta\in (0,1]$.

\begin{proof}[Proof of Theorem \ref{thm. mean square X to zero}]
We begin with $(\mathbf{A}) \implies (\mathbf{B})$.\\
\newline
   Let $Z(t)=(r^2 \ast \mathbb{E}[Y^2])(t)$. Then, as pointed out before we have     \begin{align*}
        \mathbb{E}[X^2(t)] & = \mathbb{E}[x^2(t)]+Z(t),\quad t \geq 0,\\
        Z(t) & = \left(r^2 \ast \mathbb{E}\left[(g+G(x_\cdot))^2\right]\right)(t)+(G^2(r_\cdot) \ast Z)(t), \quad t \geq 0.
    \end{align*}
    Define $x_1(t)=0$ for $t\leq 0$ and $x_1(t)=(r\ast f)(t)$ for $t\geq 0$.
    Then 
\begin{equation} \label{eq. x_1 dynamics}
    x'_1(t) = \int_{[-\tau,0]}x_1(t+u)\nu(du)+f(t), \quad  t > 0.   
\end{equation}
Next, let $u$ obey equation \eqref{eq. perturbed ODE}, extending $u$ to be zero on $[-\tau,0]$. Define 
\begin{equation} \label{eq. delta}
    \delta(t) \coloneqq  x_1(t)-u(t), \quad t\geq -\tau.
\end{equation}  
Then $\delta$ obeys 
\begin{align*}
    \delta'(t) & = x'_1(t)-u'(t) 
     = (x_1 \ast \nu)(t) + u(t)\\
    & = (\delta\ast\nu)(t)+(\tilde{\nu}\ast u)(t)+u(t),
\end{align*}
where $\tilde{\nu}$ is a finite measure on $[0,\infty)$ constructed from $\nu$ as in \eqref{eq. tildemu}. The condition \textbf{(A)}(iii) ensures we can apply Lemma \ref{lem. perturbed ODE} so that $u(t) \to 0$ as $ t \to \infty$: this along with $\nu$ being a finite measure ensures $v(t) \coloneqq (\tilde{\nu} \ast u)(t)+u(t) \to 0$ as $ t \to \infty$. For results regarding the convolutions of finite measures on $[0,\infty)$ with functions see \cite[Section 3.2 \& 3.6]{GLS}; these will be used extensively in this argument as well as in subsequent proofs. As $\delta$ obeys a linear functional differential equation with zero initial function, it obeys a variation of constants formula given by
\[
\delta(t)=(r \ast v )(t).
\]
 Once again appealing to the finiteness of $\nu$ along with condition \textbf{(A)}(i), we see that $\delta(t) \to 0$ as $t \to \infty$. But $x_1(t)=\delta(t)+u(t)$; hence we have shown $x_1(t) \to 0$ as $t\to\infty$. This gives  $\mathbb{E}[x^2(t)] \to 0$ as $ t \to \infty$. To see this, consider the bound
\[
\mathbb{E}[x^2(t)] \leq 2\mathbb{E}[x_0^2(t)]+ 2x_1^2(t).
\]
We have shown $x_1(t) \to 0$ as $t\to\infty$ and we know condition \textbf{(A)}(i) ensures $\mathbb{E}[x_0^2(t)]\leq Ce^{-2\alpha t}$ for some $\alpha<-v_0(\nu)$. Hence $\mathbb{E}[x^2(t)] \to 0$ as $ t \to \infty$, as claimed. Next we want to show $Z(t) \to 0$ as $t\to\infty$. Recall
\begin{equation*} 
   Z(t) = \gamma(t)+(G^2(r_\cdot) \ast Z)(t), \quad t \geq 0,
\end{equation*}
with $\gamma(t) = \left(r^2 \ast \mathbb{E}\left[(g+G(x_\cdot))^2\right]\right)(t)$. We have the immediate inequality
\[
\gamma(t) \leq \left(r^2 \ast \left(2g^2+2\mathbb{E}[G^2(x_\cdot)]\right) \right)(t).
\]
Let $u_1$ solve
\[
u'_1(t)=-2\alpha u_1(t)+g^2(t), \quad t \geq 0,
\]
where $\alpha > 0$ is chosen such that $|r(t)|\leq Ke^{-\alpha t}$ for all $t\geq 0$ and some constant $K>0$ (condition \textbf{(A)}(i) ensures we can always do this). Then condition \textbf{(A)}(iv) allows us to apply the method of Lemma \ref{lem. perturbed ODE}, with $g^2$ in the role of $f$, to ensure that $u_1(t) \to 0$ as $t\to\infty$. Thus we have
\[
(r^2 \ast g^2)(t)=\int_0^t r^2(t-s)g^2(s)ds \leq K^2 \int_0^t e^{-2\alpha(t-s)}g^2(s)ds,
\]
and noticing that the term on the righthand side is exactly $K^2 u_1(t)$ shows that $(r^2 \ast g^2)(t) \to 0$ as $t \to \infty$. Now we need only show $\mathbb{E}[G^2(x_t)] \to 0$ as $t\to\infty$ which will ensure $\gamma(t) \to 0$ as $t\to\infty$. Notice in the case when $\mu$ is zero that this is automatically true, and that we can conclude that $\gamma(t)\to 0$ as $t\to\infty$ directly. Dealing with the case of non--trivial $\mu$, observe that $G(x_t)=G([x_0]_t)+G([x_1]_t)$
which yields the inequality 
\[
G^2(x_t) \leq 2G^2([x_0]_t)+2G^2([x_1]_t).
\]
Now
\begin{equation*}
  | G([x_0]_t) | \leq  \int_{[-\tau,0]}|x_0(t+u)| |\mu|(du),
\end{equation*}
and so
\begin{align*}
  G^2([x_0]_t) & \leq \left( \int_{[-\tau,0]} |x_0(t+u)| |\mu|(du) \right)^2\\
  & = \int_{[-\tau,0]}\int_{[-\tau,0]}|x_0(t+s)|\cdot|x_0(t+u)| |\mu|(du)|\mu|(ds)\\
  & \leq \int_{[-\tau,0]}\int_{[-\tau,0]} \left(\frac{1}{2}x_0^2(t+s)+\frac{1}{2}x_0^2(t+u)\right) |\mu|(du) |\mu|(ds)\\
  & = |\mu|\left([-\tau,0]\right) \cdot \int_{[-\tau,0]}x_0^2(t+s)|\mu|(ds).
\end{align*}
Thus, by taking expectations and using the exponential estimate $\mathbb{E}[x_0^2(t)] \leq C(\psi)e^{-2\alpha t}$ implied by condition \textbf{(A)}(i), we see that
\begin{align*}
   \mathbb{E}[G^2([x_0]_t)] & \leq C(\psi)e^{-2\alpha t}|\mu|\left([-\tau,0]\right) \cdot \int_{[-\tau,0]}e^{-2\alpha s}|\mu|(ds) \\
   & \leq K(\psi)e^{-2\alpha t},
\end{align*}
for some constant $K$. Thus we have $\mathbb{E}[G^2([x_0]_t)] \to 0$ as $t\to\infty$. Next, write $\tilde{\mu}$ in terms of $\mu$ as in \eqref{eq. tildemu}. Then 
\[
G([x_1]_t) = (\tilde{\mu}\ast x_1)(t).
\]
As shown above, $x_1(t) \to 0$ as $t\to\infty$, and as $\mu\in M[0,\infty)$ is finite, we have once again that the above convolution tends to zero as $t\to\infty$. Hence $G^2([x_1]_t) \to 0$ as $t\to\infty$. This now gives us that $\gamma(t) \to 0$ as $t \to \infty$. 

Next we make use of a variation of constants formula obeyed by $Z$ (see \cite[Thm. 2.3.5]{GLS}):
\[
Z(t)=\gamma(t)+(\gamma \ast \rho)(t),
\]
where $\rho$ is defined by \eqref{eq. rho}. Conditions \textbf{(A)} (i) and (ii) satisfy the assumptions of Lemma~\ref{lem. exponential integrability of rho} and thus $\rho \in L^1(\mathbb{R}_+)$. This along with $\gamma(t) \to 0$ as $t\to\infty$ ensures that $Z(t) \to 0$ as $t \to \infty$ which completes the proof of the forward implication $(\textbf{A}) \implies (\textbf{B})$. 

In the case when $\mu$ is trivial, we note that 
\[
\mathbb{E}[X^2(t)]=\mathbb{E}[x^2(t)]+(r^2\ast g^2)(t), \quad t\geq 0.
\]
We have shown above that both terms on the right hand side tend to zero as $t\to\infty$, so  the implication $(\textbf{A}) \implies (\textbf{B})$ proven in this case also. 

We now show $(\textbf{B}) \implies (\textbf{A})$.\\
\newline
Step 1: $(\textbf{B}) \implies \textbf{(A)}(i):$\\
\newline
Choose $\psi_1 \in C([-\tau,0];\mathbb{R})$ arbitrary and deterministic and $\psi_2=0$. Equation \eqref{eq. Mean Sqaure X} immediately tells us that $\mathbb{E}[X^2(t,\psi)] \geq \mathbb{E}[x^2(t,\psi)]$ for all $\psi$. Thus for deterministic $\psi$ we have $x^2(t,\psi)\to 0$ as $t\to\infty$. In particular, $x^2(t,\psi_1) \to 0$ and $x^2(t,\psi_2) \to 0$ as $t \to \infty$. Consider now
$\tilde{x}(t,\psi) \coloneqq x(t,\psi_1)-x(t,\psi_2)$ for $t\geq -\tau$. Thus 
\begin{align*}
    \tilde{x}(t,\psi)  = x_0(t,\psi_1)-x_0(t,\psi_2)  = x_0(t,\psi_1),
\end{align*}
as $x_0(t,0)=0$ for all $t$. Hence we have
\begin{align*}
    x_0^2(t,\psi_1)=\tilde{x}^2(t,\psi) & =(x(t,\psi_1)-x(t,\psi_2))^2 \leq 2x^2(t,\psi_1)+2x^2(t,\psi_2),
\end{align*}
which guarantees $x_0(t,\psi) \to 0$ as $t\to\infty$ for all deterministic $\psi$. But as pointed out earlier, this implies that $v_0(\nu)<0$ which implies that $r(t)\to 0$ as $t\to\infty$, proving \textbf{(A)}(i).
\\
\newline
Step 2: $(\textbf{B}) \implies \textbf{(A)}(iii):$\\
\newline
Since $x(t,\psi) \to 0$ as $t\to\infty$ when $\psi =0$, we must also have $x_1(t) \to 0$. We now let $u$ solve equation \eqref{eq. perturbed ODE} and introduce $\delta$ as defined by \eqref{eq. delta}. We have 
\begin{align*}
    \delta'(t)  = (x_1\ast \nu)(t) + u(t) = (\tilde{\nu}\ast x_1)(t) + x_1(t) - \delta(t),
\end{align*}
where $\tilde{\nu}$ is constructed from $\nu$ as in \eqref{eq. tildemu}.  
If we define $v(t) := (\tilde{\nu} \ast x_1)(t) + x_1(t)$, we have that $v(t) \to 0$ as $t\to\infty$ by virtue of the fact that $x_1(t) \to 0$ as $t\to\infty$ and $\tilde{\nu}$ being finite. Thus, as $t\to\infty$ we have that
\[
\delta(t) = \int_0^t e^{-(t-s)}v(s)ds \to 0,
\]
and in turn by the definition of $\delta(t)$ we have that $u(t) \to 0$ as $t\to\infty$. By the converse half of 
Lemma~\ref{lem. perturbed ODE}, we have that $\int_t^{t+\delta} f(s)\,ds \to 0$ as $t\to\infty$ for all $\delta\in (0,1]$, and hence condition \textbf{(A)}(iii) is proven.\\
\newline
Step 3: $(\textbf{B}) \implies \textbf{(A)}(iv):$\\
\newline
Let $Z$ be defined as above. Since we have $\mathbb{E}[X^2(t)] \geq Z(t)$, we automatically have $Z(t) \to 0$ as $t\to\infty$, and in the same way we also have $\gamma(t) \to 0$ as $t\to\infty$ for all $\psi$. Now fix $\psi$ to be deterministic: then $\gamma(t)\to 0$ as $t\to\infty$ implies 
\[
\int_0^t r^2(t-s)\left(g(s)+G(x_s)\right)^2ds \to 0, \quad t\to\infty.
\]
The continuity of $r$ (and the fact $r(0)=1$) ensures that for all $k \in [0,1)$, there exists an $\eta_k$ such that $r^2(t) \geq k$ for all $t \in [0,\eta_k]$. Thus for $t\geq \eta_k$ we have 
\[
\int_{t-\eta_k}^t r^2(t-s)\left(g(s)+G(x_s)\right)^2ds \geq \int_{t-\eta_k}^t k\left(g(s)+G(x_s)\right)^2ds \geq 0.
\]
But we also have that
\[
\int_{t-\eta_k}^t r^2(t-s)\left(g(s)+G(x_s)\right)^2ds \leq \int_0^t r^2(t-s)\left(g(s)+G(x_s)\right)^2ds,
\]
so we must have
\[
\int_{t-\eta_k}^t \left(g(s)+G(x_s)\right)^2ds \to 0, \text{ as } t \to \infty.
\]
If we can select a $k \in [0,1)$ such that $\eta_k = 1$, we arrive at 
\[
\int_{t-1}^t \left(g(s)+G(x_s)\right)^2ds \to 0, \quad t\to\infty.
\]
If no such $k$ can be selected, we proceed as follows. We have for some $k \in [0,1)$ that $\eta_k <1$ (othwerwise there is nothing to prove) and that
\begin{align*}
    \int_{t-\eta_k}^t \left(g(s)+G(x_s)\right)^2ds \to 0, \quad t\to\infty.
    \end{align*} 
    Replacing $t$ by $t-\eta_k$ yields 
  \begin{align*}
  \int_{t-2\eta_k}^{t-\eta_k} \left(g(s)+G(x_s)\right)^2ds \to 0, \quad t\to\infty,
  \end{align*} 
    and combining these limits gives 
\begin{align*}
\int_{t-2\eta_k}^{t} \left(g(s)+G(x_s)\right)^2ds \to 0 \quad t\to\infty.
\end{align*}
We continue in this manner until we find an $n \in \mathbb{N}$ such that $n\eta_k>1$: this then implies 
\begin{align*}
    \int_{t-1}^t \left(g(s)+G(x_s)\right)^2ds & \leq \int_{t-n\eta_k}^t \left(g(s)+G(x_s)\right)^2ds\\
    & = \sum_{j=1}^n \int_{t-j\eta_k}^{t-(j-1)\eta_k} \left(g(s)+G(x_s)\right)^2ds.
\end{align*}
Passing to the limit we see
\[
\int_{t-1}^t \left(g(s)+G(x_s)\right)^2ds \to 0, \text{ as } t \to \infty,
\]
which holds irrespective of the value of $\eta_k$. Notice in the case when $\mu$ is trivial that $G(x_t)=0$ for all $t\geq 0$, and so \textbf{(A)} (iv) holds automatically.

We continue now in the case when $\mu$ is non--trivial, in which case we cannot expect $G(x_t)$ to automatically be zero.
We have already shown that \textbf{(B)} implies $x(t) \to 0$ as $t\to\infty$, which in turn ensures $G(x_t) \to 0$ as $t\to\infty$. Thus $\int_{t-1}^tG^2(x_s)ds \to 0$ as $t\to\infty$, so we get
\begin{equation}\label{eq.conv1}
\int_{t-1}^tg^2(s)ds + \int_{t-1}^t2g(s)G(x_s)ds \to 0, \quad t\to\infty.
\end{equation}
Note by the Cauchy--Schwarz inequality that we have
\begin{equation} \label{eq.conv2}
\left| \int_{t-1}^t2g(s)G(x_s)ds\right|^2 \leq \int_{t-1}^tg^2(s)ds\int_{t-1}^t4G^2(x_s)ds,
\end{equation}
and that the second integral on the right hand side tends to zero as $t\to\infty$. Thus, using \eqref{eq.conv1} and \eqref{eq.conv2}, for all $\epsilon>0$, there exists $T_1(\epsilon)$ and $T_2(\epsilon)$ such that for all $t> T(\epsilon):=\max\{T_1(\epsilon),T_2(\epsilon)\}$ we have
\[
\left|\int_{t-1}^tg^2(s)ds + \int_{t-1}^t2g(s)G(x_s)ds \right| < \epsilon \text{ and } \left|\int_{t-1}^t2g(s)G(x_s)ds \right| < \epsilon \cdot\sqrt{\int_{t-1}^tg^2(s)ds}.
\]
Hence for $t\geq T(\epsilon)$, we have 
\begin{align*}
    \int_{t-1}^tg^2(s)ds & = \left| \int_{t-1}^tg^2(s)ds+\int_{t-1}^t2g(s)G(x_s)ds-\int_{t-1}^t2g(s)G(x_s)ds\right|\\
    &  < \epsilon+ \epsilon \cdot\sqrt{\int_{t-1}^tg^2(s)ds}.
\end{align*}
Let $p(x)\coloneqq x^2-\epsilon x - \epsilon$. With $x\coloneqq \sqrt{\int_{t-1}^tg^2(s)ds}\geq 0$, we have $p(x)<0$. This implies 
\[
|x| < \dfrac{\epsilon+\sqrt{\epsilon^2+4\epsilon}}{2}.
\]
Hence 
\[
\int_{t-1}^tg^2(s)ds  < \dfrac{\epsilon+\sqrt{\epsilon^2+4\epsilon}}{2}, \quad t\geq T(\epsilon).
\]
Since $\epsilon$ is arbitrary, we have $\int_{t-1}^tg^2(s)ds \to 0$ as $ t \to \infty$, which is condition \textbf{(A)}(iv). \\
\newline
Step 4: $(\textbf{B}) \implies \textbf{(A)}(ii):$\\
\newline
The first case we must consider is when the measure $\mu$ is the zero measure. This gives $G(r_{\cdot}) \equiv 0$ so that $\Vert G(r_{\cdot})\Vert_{L^2(\mathbb{R}_+)}=0$ is automatically less than one, and \textbf{(A)}(ii) automatically holds. 

From here we exclude the case where $\mu$ is the zero measure. Since \textbf{(B)} holds, we have that $r(t)\to 0$ as $t\to\infty$, and indeed that $r\in L^1(\mathbb{R}_+)$. Since $\mu$ is a finite measure, we therefore have that $G(r_\cdot)$ is in $L^1(\mathbb{R}_+)$. Also, the fact that $r(t)\to 0$ as $t\to\infty$ implies that $G(r_t)\to 0$ as $t\to\infty$. Therefore, we have that $G^2(r_\cdot)\in L^2(\mathbb{R}_+)$. 
Hence $\Vert G(r_{\cdot})\Vert_{L^2(\mathbb{R}_+)}<+\infty$. On the other hand, since we are now considering non--trivial $\mu$, our earlier arguments show that  
$\Vert G(r_{\cdot})\Vert_{L^2(\mathbb{R}_+)}>0$. 

In the proof of Step 3, we deduced that $\gamma(t) \to 0$ as $t\to\infty$ so our first task will be to show that $\gamma(t)$ is in fact strictly positive on some non--trivial interval. There are two cases we must consider.\\
\newline
\textbf{Case 1}: $g(t)+G([x_1]_t) \not\equiv 0$ for all $t\geq0$.\\
\newline
Take $\psi \equiv 0$ so that
\begin{align*}
    g(t)+G(x_t)  = g(t)+G([x_0]_t)+G([x_1]_t)  = g(t) + G([x_1]_t) \eqqcolon \gamma_1(t),
\end{align*}
and hence $\gamma(t) = (r^2 \ast \gamma_1^2)(t)$ for $t \geq 0$. Note as $\gamma_1$ is continuous (and not identically equal to zero) there exists an interval $(t_1,t_2) \subset [0,\infty)$ such that $\gamma_1^2(t)\geq \eta$ for all $t \in (t_1,t_2)$ for some $\eta>0$. Let $t\geq \theta$, where $\theta>0$ will be chosen later: then
\begin{align*}
    \gamma(t)=\int_0^t r^2(t-s)\gamma_1^2(s)ds  \geq \int_{t-\theta}^tr^2(t-s)\gamma_1^2(s)ds
     \geq \inf_{u\in [0,\theta]}r^2(u) \int_{t-\theta}^t\gamma_1^2(s)ds.
\end{align*}
Now we choose $\theta$ small enough such that $t_1+\theta<t_2$
and that $\inf_{u\in [0,\theta]}r^2(u) \geq \frac{1}{2}$.
Now choose $t$ such that $t \in (t_1+\theta,t_2)$, which means we have,
\[
\gamma(t)\geq \inf_{u\in [0,\theta]}r^2(u)\cdot \eta \theta \geq \dfrac{\eta \theta}{2}.
\]

\textbf{Case 2}: $g(t)+G([x_1]_t) \equiv 0$ for all $t\geq0$.\\
\newline
Since $\mu$ is non--zero, we may choose $\psi$ so that $\int_{[-\tau,0]}\mu(dt)\psi(t) \neq 0$. This ensures $G([x_0]_0)\neq0$. Then we get
\begin{align*}
       g(t)+G(x_t) = g(t)+G([x_0]_t)+G([x_1]_t)
    & = G([x_0]_t) \eqqcolon \gamma_2(t).
\end{align*}
The continuity of $\gamma_2$ and the fact $\gamma_2(0) \neq 0$ ensures $\gamma_2^2(t)>\eta'$ for all $t \in[0,t_2')$ for some $t_2'>0$. Thus by a similar argument as in \textbf{Case 1} we must have $\gamma$ strictly positive on some non--trivial interval.\\
\newline
Thus we have concluded when $\mu \not\equiv 0$, we can find a deterministic $\psi$ such that there exists an $\tilde{\eta} >0$ and an interval $(t_1',t_2')$ so that $\gamma(t,\psi)\geq \tilde{\eta} >0$ for all $t\in (t_1',t_2')$. Recall  $Z(t)=\gamma(t)+(\gamma \ast \rho)(t)$, where $\rho$ is defined as in \eqref{eq. rho}. Now suppose $\Vert G(r_{\cdot})\Vert_{L^2(\mathbb{R}_+)} \geq 1$: by the Renewal Theorems 3.1.4 and 3.1.5 in Alsmeyer \cite{Als}, there exists a $\lambda \geq 0$ such that $\rho(t)/e^{\lambda t} \to c>0$ as $t \to \infty$. Let $T>0$ be arbitrary and choose $t\geq T$. Then
\begin{align*}
    \dfrac{Z(t)}{e^{\lambda t}}  \geq \frac{1}{e^{\lambda t}}\int_0^T \rho(t-s)\gamma(s)ds
     = \int_0^T \left[ \frac{\rho(t-s)}{e^{\lambda(t-s)}}-c\right]e^{-\lambda s}\gamma(s) ds+ c\int_0^T e^{-\lambda s}\gamma(s) ds.
\end{align*}
Hence
\[
\liminf_{t \to \infty}e^{-\lambda t}Z(t) \geq c\int_0^T e^{-\lambda s}\gamma(s) ds.
\]
By hypothesis, $Z(t) \to 0$ as $t\to\infty$. Using this and the fact that $c>0$, we must have
\[
\int_0^T e^{-\lambda s}\gamma(s) ds=0, \text{ for all } T>0.
\]
But since $\gamma$ is strictly positive on a non--trivial interval, the above integral cannot be zero and hence we reach our desired contradiction. Thus $\Vert G(r_{\cdot})\Vert_{L^2(\mathbb{R}_+)} < 1$ which proves condition \textbf{(A)}(ii) and hence the reverse implication \textbf{(B)} implies \textbf{(A)}. Since we already proved that \textbf{(A)} implies \textbf{(B)}, the proof is complete.
\end{proof}

\begin{proof}[Proof of Theorem \ref{thm. mean square X exponential decay}]
 We first show $(\textbf{A}) \implies (\textbf{B})$.\\
\newline
With $Z(t,\psi)$ defined as in the proof of Theorem \ref{thm. mean square X to zero} we have $\mathbb{E}[X^2(t)]=\mathbb{E}[x^2(t)]+Z(t)$. The first object we study is $\mathbb{E}[x^2(t)]$, which can be estimated by
$\mathbb{E}[x^2(t)] \leq 2\mathbb{E}[x_0^2(t)]+2x_1^2(t)$.
Condition \textbf{(A)}(i) ensures $\mathbb{E}[x_0^2(t)]\leq C(\psi)e^{-2\alpha t}$ where $\alpha>0$ is such that $|r(t)|\leq Ke^{-\alpha t}$ for $t\geq 0$ and some constant $K>0$. Next we let $u$ be the solution of
\begin{equation} \label{eq.uinexpthm}
u'(t)=-\beta_2u'(t)+f(t), \quad t \geq 0,
\end{equation}
with $u(0)=0$, where $\beta_2$ is chosen such that condition \textbf{(A)}(iv) holds. We can now estimate $u(t)$ for $t\geq 0$:
\begin{align*}
    |u(t)| \leq e^{-\beta_2t}\left| \int_0^t e^{\beta_2 s}f(s)ds\right|  \leq C(f) e^{-\beta_2t}.
\end{align*}
Define $\delta=x_1-u$ as in \eqref{eq. delta}, where $u$ is understood now to solve \eqref{eq.uinexpthm}. Then $\delta(t)=0$ for $t\leq 0$. Write $\tilde{\nu}$ in terms of $\nu$ as in \eqref{eq. tildemu} so that  
\begin{align*}
    \delta'(t)  = (x_1\ast \nu)(t)+\beta_2u(t)
    = (\delta\ast \nu)(t)+(\tilde{\nu} \ast u)+\beta_2u(t).
\end{align*}
With $v(t):= (\tilde{\nu}\ast u)+\beta_2u(t)$, $\delta$ obeys the variation of constants formula 
\begin{align*}
  \delta(t)  = \int_0^t v(t-s)r(s)ds, \quad t\geq 0.
\end{align*}
 Observe that $v$ depends only on $u$ and $\nu$ and so it too obeys an exponential estimate of the form $|v(t)|\leq C(f)e^{- \beta_2t}$ for $t\geq 0$. Combining this with the  exponential estimate on $r$ and the equation above for $\delta$ we obtain
\begin{align*}
    \left|\delta(t)\right| \leq C(f)Ke^{-\min(\alpha-\epsilon,\beta_2) t}, \quad t\geq 0,
\end{align*}
for arbitrarily small $\epsilon<\alpha$. This estimate follows from 
\[
|\delta(t)|\leq C(f)K \int_0^t e^{-\beta_2(t-s)} e^{-\alpha s}\,ds,
\]
and by estimating the integral in the cases $\alpha\geq \beta_2$ and $\alpha <\beta_2$.  Now using equation \eqref{eq. delta} we can combine our estimates on $\delta$ and $u$ to get an estimate on $x_1$. Thus we have proved the estimate
\[
|x_1(t)|\leq C'(f)e^{-\min(\alpha-\epsilon,\beta_2) t},\quad t\geq 0.
\]
Notice that this also gives the estimate 
\[
|G([x_1]_t)|\leq C_2(f) e^{-\min(\alpha-\epsilon,\beta_2) t},\quad t\geq 0.
\]
Putting together the estimates for $x_1$ and $\mathbb{E}[x_0^2]$ yields  
\[
\mathbb{E}[x^2(t,\psi)]\leq C(f,\psi)e^{-2\min(\alpha-\epsilon,\beta_2)t}, \quad t\geq 0.
\]

Next we focus on $Z(t)$, let $\rho$ and $\gamma$ be defined as in \eqref{eq. rho} and \eqref{eq. gamma} respectively.
We first estimate $\gamma$:
\begin{equation} \label{eq.gammaexpest}
    \gamma(t) \leq 2\int_0^t r^2(t-s)g^2(s)ds+2\int_0^tr^2(t-s)\mathbb{E}[G^2(x_s)]ds.
\end{equation}
Considering the first term on the righthand side of \eqref{eq.gammaexpest}, condition \textbf{(A)}(iii) ensures there exists a $\beta_1$ such that
\[
\int_0^\infty e^{2\beta s}g^2(s)ds< \infty, \text{ for all } \beta \leq \beta_1.
\]
Thus 
\begin{equation} \label{eq.r^2astg^2exp}
\int_0^t r^2(t-s)g^2(s)ds \leq K^2 e^{-2\alpha t}\int_0^t e^{2\alpha s}g^2(s)\,ds.
\end{equation}
When $\alpha\leq \beta_1$, the integral is uniformly bounded, and the first term grows no faster than $e^{-2\alpha t}$. When 
$\alpha>\beta_1$, we have 
\begin{align*}
\int_0^t r^2(t-s)g^2(s)ds &\leq K^2 e^{-2\alpha t} \int_0^t e^{2\alpha s}g^2(s)\,ds 
= K^2 e^{-2\alpha t} \int_0^t e^{2(\alpha-\beta_1)s} e^{2\beta_1  s}g^2(s)\,ds \\
&\leq K^2 e^{-2\alpha t} e^{2(\alpha-\beta_1)t} \int_0^t  e^{2\beta_1  s}g^2(s)\,ds
\leq K^2 e^{-2\alpha t} e^{-2\beta_1 t} \int_0^\infty  e^{2\beta_1  s}g^2(s)\,ds,
\end{align*}
so that the integral is $O(e^{-2\beta_1 t})$ as $t\to\infty$. Thus, we have that $(r^2\ast g^2)(t) = O(e^{-2\min(\alpha,\beta_1)t})$ as $t\to\infty$. 

At this moment, we have enough information to conclude the proof in the case when $\mu$ is zero, so we halt the general argument to dispense with this trivial case. We have already obtained the estimate 
\[
\mathbb{E}[x^2(t,\psi)]\leq C(f,\psi)e^{-2\min(\alpha-\epsilon,\beta_2)t}, \quad t\geq 0,
\]
and we have just shown that $(r^2\ast g^2)(t) \leq C(g) e^{-2\min(\alpha,\beta_1)t}$. Therefore 
\[
\mathbb{E}[X^2(t)]=\mathbb{E}[x^2(t)]+(r^2\ast g^2)(t)
\leq C(f,g,\psi)e^{-2\min(\alpha-\epsilon,\beta_2,\beta_1)t}, \quad t\geq 0.
\]
The exponent on the righthand side depends on $f$ and $g$ through $\beta_2$ and $\beta_1$, but there is no dependence in the exponent in $\psi$.
 
In the rest of the proof, we concentrate on the case where $\mu$ is non--trivial.

To control the second term on the righthand side of \eqref{eq.gammaexpest} we need an estimate on $\mathbb{E}[G^2(x_t)]$; one can get this by following an identical argument in the proof of Theorem~\ref{thm. mean square X to zero}. We have 
\begin{align*}
\mathbb{E}[G^2([x_0]_t)]  \leq K(\psi)e^{-2\alpha t}, \quad t\geq 0.
\end{align*}
Using the estimate earlier obtained for $G([x_1]_t)$, we get 
\[
G^2([x_1]_t)\leq C_2^2(f) e^{-2\min(\alpha-\epsilon,\beta_2) t},\quad t\geq 0.
\]
This implies that 
\[
\mathbb{E}[G^2(x_t)]\leq C(f,\psi) e^{-2\min(\alpha-\epsilon,\beta_2) t}, \quad t\geq 0.
\]
Therefore
\begin{align*}
\int_0^t r^2(t-s) \mathbb{E}[G^2(x_s)]\,ds 
&\leq C(f,\psi) K^2 \int_0^t e^{-2\alpha(t-s)} e^{-2\min(\alpha-\epsilon,\beta_2) s}\,ds\\
&\leq C'(f,\psi) e^{-2\min(\alpha-\epsilon,\beta_2) t}.
\end{align*}
Hence we have
\[
|\gamma(t)| \leq C(f,g,\psi)e^{-2\alpha(f,g)t}, \quad t\geq 0,
\]
where $\alpha(f,g):=\min(\alpha-\epsilon,\beta_1,\beta_2)$. The $f$ and $g$ dependence here comes from the $f$ and $g$ dependence on $\beta_2$ and $\beta_1$ respectively. 

Since we are now tackling the case when $\mu$ is non--zero, we note that conditions 
 \textbf{(A)}(i) and \textbf{(A)}(ii) can be used to apply Lemma~\ref{lem. exponential integrability of rho}, so we have
\[
\int_0^{\infty}e^{2\lambda s}\rho(s)ds =  \dfrac{\int_0^{\infty}G_\lambda^2(r_s)ds}{1-\int_0^{\infty}G_\lambda^2(r_s)ds} \eqqcolon K',
\]
which is finite for any $\lambda \in [0,\alpha')$ where $\Gamma(\alpha')=1$.
Hence set $\lambda:=\min\{\alpha(f,g),\alpha'-\epsilon\}$ for an arbitrarily small $\epsilon$. Here the choice of $\lambda$ clearly depends on $f$ and $g$, while $\alpha'$ depends on $\nu$ and $\mu$, but not on $\psi$. Since $\lambda\leq \alpha$, we have
\begin{align*}
    Z(t)& =\int_0^t\gamma(t-s)\rho(s)ds+\gamma(t)\\ 
    & \leq C(f,g,\psi)\int_0^t e^{-2\lambda(t-s)}\rho(s)ds+C(f,g,\psi)e^{-2\lambda t}\\
    & = C(f,g,\psi)e^{-2\lambda t}\left(\int_0^\infty e^{2\lambda s}\rho(s)ds+1 \right)\\
    & \leq C'(f,g,\psi)e^{-2\lambda t},
\end{align*}
where the finiteness of the integral at the penultimate step follows from $\lambda<\alpha'$. Lastly, we have 
\[
\mathbb{E}[X^2(t)]=\mathbb{E}[x^2(t)]+Z(t)
\leq C(f,\psi)e^{-2\min(\alpha-\epsilon,\beta_2)t}+ C'(f,g,\psi)e^{-2\lambda t}
\leq C''(f,g,\psi)e^{-2\lambda t}.
\]  
Since $\lambda>0$ depends on $f$ and $g$, but not on $\psi$, we have completed the proof that \textbf{(A)} implies \textbf{(B)}. \\
\newline
We now show $(\textbf{B}) \implies (\textbf{A}):$\\
\newline
Step 1: $(\textbf{B}) \implies$ \textbf{(A)}(i) and \textbf{(A)}(ii):\\
\newline
By hypothesis we have $\mathbb{E}[X^2(t,\psi)] \to 0$ as $t\to\infty$ for all $\psi$ and so Theorem \ref{thm. mean square X to zero} implies conditions \textbf{(A)}(i) and \textbf{(A)}(ii).\\
\newline
Step 2: $(\textbf{B}) \implies$ \textbf{(A)}(iv):\\
\newline
Using the fact that $\mathbb{E}[X^2(t,\psi)] \geq \mathbb{E}[x^2(t,\psi)]$, if we choose $\psi$ to be deterministic we obtain $x^2(t,\psi)\leq C(f,g,\psi)e^{-2\alpha t}$ for all $t\geq 0$ (here $\alpha=\alpha(f,g)$). Using equation \eqref{eq. x in terms of x_0 and x_1} and setting $\psi=0$ gives
$x_1^2(t)=x^2(t,0)\leq C(f,g,0)e^{-2\alpha t}$ for $t\geq 0$, 
which implies $|x_1(t)| \leq Ce^{-\alpha t}$ for $ t\geq 0$. 
Let $u$ be the solution to
\[
u'(t)=-\beta_2 u(t)+f(t), \quad t \geq 0,
\]
with $u(t)=0$ for all $t\leq 0$ and  $\beta_2 \in (0,\alpha)$.
Extend $x_1$ to be zero for $t<0$ and define $\delta$ as in \eqref{eq. delta}. Using the fact that $x_1$ obeys \eqref{eq. x_1 dynamics}, we see that for $t>0$
\begin{align*}
    \delta'(t) =\beta_2u(t)+( x_1\ast \nu)(t)
     =-\beta_2\delta(t)+\beta_2x_1(t)+(x_1\ast \nu)(t),
\end{align*}
Define $v(t):= \beta_2x_1(t)+(x_1\ast \nu)(t)$ for $t\geq 0$. We note $|v(t)| \leq Ce^{-\alpha t}$ for all $t \geq 0$ by virtue of the estimate on $x_1$ above and $\nu$ being finite. Solving for the above equation for $\delta$ gives
\[
|\delta(t)| =\left| \int_0^t e^{-\beta_2(t-s)}v(s)ds\right| \leq Ce^{-\beta_2t}\int_0^\infty e^{-(\alpha-\beta_2)s}ds.
\]
The integral on the righthand side is finite as $\beta_2 \in (0,\alpha)$. Thus we have $|\delta(t)| \leq C'e^{-\beta_2t}$ for all $t\geq 0$. Once again using \eqref{eq. delta} and the estimates obtained for $x_1$ and $\delta$ we see that $|u(t)| \leq (C+C') e^{-\beta_2t}$ for $t\geq 0$. But
\[
u(t) = \int_0^t e^{-\beta_2(t-s)s}f(s)ds,
\]
and so
\begin{align*}
    \left| \int_0^t e^{-\beta_2(t-s)}f(s)ds \right| \leq (C+C') e^{-\beta_2t}, \quad t\geq 0
    \end{align*}
    which implies 
   \begin{align*} 
    \left| \int_0^t e^{\beta_2s}f(s)ds \right| \leq C+C'=:B, \quad t\geq 0.
\end{align*}
Thus condition \textbf{(A)}(iv) is proven.\\
\newline
Step 3: $\textbf{(B)} \implies$ \textbf{(A)}(iii):\\
\newline
Take $\psi$ to be deterministic. From \textbf{(B)}, and using definitions \eqref{eq. Z} and \eqref{eq. Z in terms of gamma}, we have straight away that $Ce^{-2\alpha t} \geq 
\mathbb{E}[X^2(t,\psi)] \geq Z(t,\psi) \geq \gamma(t,\psi)$ for $t\geq 0$, where $\alpha=\alpha(f,g)>0$. Since $\psi$ is deterministic, this gives  
\[
\int_0^t r^2(t-s) \left( g(s)+G(x_s)\right)^2ds \leq C e^{-2\alpha t}, \quad t\geq 0.
\]
The continuity of $r$ and the fact $r(0)=1$ ensure that there exists an $\eta \in (0,1)$ such that $r^2(t) \geq \frac{1}{2}$ for all $t \in [0,\eta)$. Thus for $t \geq \eta$, we have 
\begin{align*}
    C e^{-2\alpha t}  \geq \int_0^t r^2(t-s) \left( g(s)+G(x_s)\right)^2ds
     \geq \frac{1}{2} \int_{t-\eta}^t \left( g(s)+G(x_s)\right)^2ds.
\end{align*}
Now let $m \in \mathbb{N}$ be the minimal integer such that $m\eta\geq1$. Thus for all $t \geq m\eta \coloneqq T'$ we have
\begin{align*}
    \int_{t-1}^t \left( g(s)+G(x_s)\right)^2ds & \leq \int_{t-m\eta}^t \left( g(s)+G(x_s)\right)^2ds
     = \sum_{j=0}^{m-1} \int_{t-(j+1)\eta}^{t-j\eta} \left( g(s)+G(x_s)\right)^2ds\\
    & \leq \sum_{j=0}^{m-1} 2C e^{-2\alpha (t-j\eta)}
     = C'e^{-2\alpha t}.
\end{align*}
We have $T' \in[1,2)$. To see this, note that $(m-1)\eta < 1$ so $T'= m\eta<1+\eta<2$. Thus the above estimate holds for all $t \geq T'$ and therefore for all $t \geq 1$ (modulo an alternative constant $C''$). In other words, we have obtained the estimate  
\[
 \int_{t-1}^t \left( g(s)+G(x_s)\right)^2ds \leq C'' e^{-2\alpha(f,g) t}, \quad t\geq 1.
\]
Next, since $Ce^{-2\alpha t}\geq \mathbb{E}[X^2(t,\psi)]\geq x^2(t,\psi)$ for $t\geq 0$, we have $|x(t,\psi)|\leq \sqrt{C}e^{-\alpha t}$ for $t\geq 0$. Thus, as $\mu$ is a finite measure on $[-\tau,0]$, $G(x_t)$ inherits an exponential estimate from $x(t)$, so that $|G(x_t)| \leq K'e^{- \alpha t}$ for some $K'>0$. Next for $t\geq 1$, we get 
\begin{align*}
    \int_{t-1}^tg^2(s)ds+\int_{t-1}^t2g(s)G(x_s)ds & \leq \int_{t-1}^tg^2(s)ds+\int_{t-1}^t2g(s)G(x_s)ds+\int_{t-1}^tG^2(x_s)ds\\
    \leq C''e^{-2\alpha t}.
\end{align*}
Clearly this yields
\[
\int_{t-1}^tg^2(s)ds-\left|\int_{t-1}^t2g(s)G(x_s)ds\right| \leq \int_{t-1}^tg^2(s)ds+\int_{t-1}^t2g(s)G(x_s)ds
\leq C''e^{-2\alpha t}.
\]
On the other hand, by the Cauchy--Schwarz inequality, and using the exponential estimate for $|G(x_t)|$, we get 
\begin{align*}
    \left| \int_{t-1}^t2g(s)G(x_s)ds \right|^2  \leq 4\int_{t-1}^tg^2(s)ds \cdot \int_{t-1}^tG^2(x_s)ds
    \leq K_2 e^{- 2\alpha t}\int_{t-1}^tg^2(s)ds,
\end{align*}
for some constant $K_2>0$. Taking the last two estimates together, this implies
\[
\int_{t-1}^tg^2(s)ds \leq C''e^{- 2\alpha t}+ \sqrt{K_2e^{- 2\alpha t}\int_{t-1}^tg^2(s)ds}, \quad t\geq 1.
\]
Write $A \coloneqq C''e^{-2\alpha t}$,  $B \coloneqq \sqrt{K_2 e^{-2\alpha t}}$, and consider $p(x)\coloneqq x^2-Bx-A$ for $x\geq 0$. Putting  $x \coloneqq \sqrt{\int_{t-1}^tg^2(s)ds}\geq 0$, we see that $p(x)\leq 0$. 
In general, the constraints $x\geq 0$ and $p(x)\leq 0$ imply 
\[
0 \leq x \leq \dfrac{B+\sqrt{B^2+4A}}{2}.
\]
Hence there is a $C_3>0$ such that
\[
\int_{t-1}^tg^2(s)ds \leq C_3e^{-2\alpha t}, \quad t\geq 1.
\]
Now let $\beta_1 < \alpha$. Then for $t\geq 1$
\begin{align*}
    \int_{t-1}^t e^{2\beta_1 s}g^2(s)ds  \leq e^{2\beta_1 t}\int_{t-1}^tg^2(s)ds \leq C_3e^{-2t(\alpha-\beta_1)}.
\end{align*}
In particular for any $n \in \mathbb{N}$,
\[
\int_{n-1}^n e^{2\beta_1 s}g^2(s)ds \leq C''e^{-2n(\alpha-\beta_1)}.
\]
Therefore 
\begin{align*}
    \int_0^\infty e^{2\beta_1 s}g^2(s)ds & \leq \sum_{n=1}^\infty C_3e^{-2n(\alpha-\beta_1)} 
    = C_3 \dfrac{e^{-2(\alpha-\beta_1)}}{1-e^{-2(\alpha-\beta_1)}}.
\end{align*}
Hence $\int_0^\infty e^{2\beta_1 s}g^2(s)ds < \infty$ for all $\beta_1 \in (0,\alpha)$. This completes the proof that \textbf{(B)} implies \textbf{(A)}, and hence both implications are proven.
\end{proof}

\begin{proof}[Proof of Theorem~\ref{thm. mean square X in L1}]

 We first show $(\textbf{A}) \implies (\textbf{B})$.\\
Using equation \eqref{eq. Z}, we need only show $\mathbb{E}[x^2], Z\in L^1(\mathbb{R}_+)$. Equation \eqref{eq. x in terms of x_0 and x_1} implies
\[
\mathbb{E}[x^2(t)] \leq 2\mathbb{E}[x_0^2(t)]+2x_1^2(t).
\]
Condition \textbf{(A)}(i) ensures $\mathbb{E}[x_0^2] \in L^1(\mathbb{R}_+)$, so we need only focus on $x_1$. Let $u$ be the solution to the differential equation in Lemma~\ref{lem. perturbed ODE}, and let $\delta$ be defined by \eqref{eq. delta} as usual. Condition \textbf{(A)} (iii) ensures $u \in L^2(\mathbb{R}_+)$. We have that $\delta'(t)=( \delta\ast \nu)(t)+v(t)$ for $t>0$ where $v=\tilde{\nu}\ast u+u$, $\delta(t)=0$ for $t\leq 0$, and $\tilde{\nu}$ is defined from $\nu$ according to \eqref{eq. tildemu}. Since $\tilde{\nu}$ is finite, $v\in L^2(\mathbb{R}_+)$. Since $\delta=r\ast v$, and $r\in L^1(\mathbb{R}_+)$ by \textbf{(A)}(i), it follows that $\delta \in L^2(\mathbb{R}_+)$. Thus by \eqref{eq. delta} we have $x_1 \in L^2(\mathbb{R}_+)$ which implies $\mathbb{E}[x^2] \in L^1(\mathbb{R}_+)$. 

We concentrate first on the case when $\mu$ is non--trivial.
Consider $Z$, recalling that if we let $\gamma$ and $\rho$ be defined as in \eqref{eq. gamma} and \eqref{eq. rho} respectively, then $Z(t)=\gamma(t)+(\gamma \ast \rho)(t)$ for $t \geq 0$.
Conditions \textbf{(A)}(i) and \textbf{(A)}(ii) ensure Lemma \ref{lem. exponential integrability of rho} is applicable and thus $\rho \in L^1(\mathbb{R}_+)$. We need only show $\gamma \in L^1(\mathbb{R}_+)$ in order to show $Z\in L^1(\mathbb{R}_+)$. We have that
\[
\gamma(t) \leq 2(r^2 \ast g^2)(t) + 2(r^2 \ast \mathbb{E}[G^2(x)])(t).
\]
Conditions \textbf{(A)}(i) and \textbf{(A)}(iv) imply the first term on the righthand side is in $L^1(\mathbb{R}_+)$: hence if we show $\mathbb{E}[G^2(x_\cdot)] \in L^1(\mathbb{R}_+)$, we are done. We estimate $\mathbb{E}[G^2(x_t)]$ as before, according to
\begin{align*}
  G^2(x_t) & \leq \int_{[-\tau,0]}\int_{[-\tau,0]} |x(t+s)||x(t+u)||\mu|(du)|\mu|(ds)\\
  & \leq \int_{[-\tau,0]}\int_{[-\tau,0]} \left(\frac{1}{2}x^2(t+s)+\frac{1}{2}x^2(t+u)\right) |\mu|(du) |\mu|(ds)\\
  &=|\mu|([-\tau,0]) \cdot \int_{[-\tau,0]} x^2(t+s)|\mu|(ds).
\end{align*}
Taking expectations gives 
\[
 \mathbb{E}[G^2(x_t)] \leq|\mu|([-\tau,0]) \cdot \int_{[-\tau,0]} \mathbb{E}[x^2(t+s)]|\mu|(ds).
\]
Since $\mathbb{E}[x^2] \in L^1(\mathbb{R}_+)$ and $|\mu|$ is a finite measure, we have  $\mathbb{E}[G^2(x_\cdot)] \in L^1(\mathbb{R}_+)$. This completes the proof of the forward implication when $\mu$ is non--trivial.

In the case when $\mu=0$, we have 
\[
\mathbb{E}[X^2(t)]=\mathbb{E}[x^2(t)]+(r^2\ast g^2)(t), \quad t\geq 0.
\]
The argument at the start of the proof guarantees that
$\mathbb{E}[x^2]\in L^1(\mathbb{R}_+)$; on the other hand,  
conditions \textbf{(A)}(i) and \textbf{(A)}(iv) imply the second term on the righthand side is in $L^1(\mathbb{R}_+)$, as above. Thus we have shown that (\textbf{B}) implies (\textbf{A}) in the case of trivial $\mu$ also.
\\
\newline
We now show $(\textbf{B}) \implies (\textbf{A}):$\\
\newline
Step 1: $(\textbf{B}) \implies$ (\textbf{A})(i):\\
\newline
Recall that $\mathbb{E}[X^2(t,\psi)] \geq \mathbb{E}[x^2(t,\psi)]$ for all $t\geq 0$ for all $\psi$. If we take $\psi$ to be deterministic then (\textbf{B}) implies 
\[
\int_0^\infty x^2(t,\psi)dt < \infty.
\]
Using \eqref{eq. x in terms of x_0 and x_1} we have that $x(t,0)=x_1(t)$, so by the last inequality $x_1 \in L^2(\mathbb{R}_+)$. On the other hand, since $x(t,\psi)-x(t,0)=x_0(t,\psi)$, we have the estimate
\begin{align*}
    x_0^2(t,\psi)  = (x(t,\psi)-x(t,0))^2 \leq 2x^2(t,\psi)+2x^2(t,0), 
\end{align*}
and so $\int_0^\infty x_0^2(t,\psi)dt < \infty$ for all $\psi$. Now let $\lambda\in \Lambda$ be such that $\text{Re}(\lambda)=v_0(\nu)$, and pick $\psi(t)=\text{Re}(e^{\lambda t})$ for $t\in[-\tau,0]$. Then 
$x_0(t,\psi)=\text{Re}(e^{\lambda t})$ for $t\geq 0$. Let $\theta=\text{Im}(\lambda)$. Then $x_0(t,\psi)=e^{v_0(\nu)t}\cos(\theta t)$ for $t\geq 0$. Hence
for arbitrary $t\geq 0$ we have
\[
\int_0^t x^2_0(s,\psi)\,ds = \int_0^t e^{2v_0(\nu)s}\cos^2(\theta s)\,ds. 
\] 
Suppose that $v_0(\nu)\geq 0$. Then the above integral diverges as $t\to\infty$, which contradicts the fact that 
$\int_0^\infty x_0^2(t,\psi)dt < \infty$ for all $\psi$. Therefore, we must have $v_0(\nu)<0$, which implies that 
 $r \in L^2(\mathbb{R}_+) $, as required.\\
\newline
Step 2: $(\textbf{B}) \implies$ \textbf{(A)}(iii):\\
\newline
Let $u$ be the solution to the differential equation in 
Lemma~\ref{lem. perturbed ODE}. Then defining $\delta$ as in \eqref{eq. delta}, and $\tilde{\nu}$ from $\nu$ as in \eqref{eq. tildemu}, it follows that
\begin{align*}
    \delta'(t)
     =-\delta(t)+x_1(t)+(\tilde{\nu} \ast x_1)(t), \quad t>0.
\end{align*}
Note also that $\delta(t)=0$ for $t\leq 0$. 
If we define $v(t) := x_1(t)+(\tilde{\nu}\ast x_1)(t)$ for $t\geq 0$, we have that $v \in L^2(\mathbb{R}_+)$ by virtue of the fact $\tilde{\nu}$ is a finite measure and $x_1 \in L^2(\mathbb{R}_+)$ (which we proved in Step 1 above). Since $\delta(t) = \int_0^t e^{-(t-s)}v(s)ds$ for $t\geq 0$, we also have that $\delta \in L^2(\mathbb{R}_+)$. Thus equation \eqref{eq. delta} implies $u \in L^2(\mathbb{R}_+)$ and since 
$u(t)=\int_0^t e^{-(t-s)}f(s)\,ds$, condition \textbf{(A)}(iii) is proven.\\
\newline
Step 3: $(\textbf{B}) \implies$ \textbf{(A)}(iv):\\
\newline
Once again let $\psi$ be deterministic. We have immediately that $\mathbb{E}[X^2(t,\psi)]\geq Z(t,\psi)$ for $t\geq 0$ and so $Z \in L^1(\mathbb{R}_+)$. Additionally, by the definition of $Z$ along with equation \eqref{eq. Z in terms of gamma} this forces $\gamma \in L^1(\mathbb{R}_+)$. Define 
\[
A(t)=\int_0^t \left(g(s)+G(x_s)\right)^2ds, \quad t\geq 0.
\]
Integrating equation \eqref{eq. gamma}, and using Fubini's theorem, we see
\begin{align*}
   \int_0^T \gamma(t)\,dt  = \int_0^T \int_0^t r^2(s)\left(g(t-s)+G(x_{t-s})\right)^2\,ds\,dt
    =  \int_0^T A(T-s) r^2(s)\,ds.
\end{align*}
Therefore we have that there is a $B>0$ such that
\[
\int_0^t A(t-s) r^2(s)\,ds \leq B, \quad t\geq 0.
\]
We need to show that $A$, which is non--negative and non--decreasing, tends to a finite limit; we already know 
from Step 1 that $r\in L^2(\mathbb{R}_+)$. Suppose to the contrary that $A(t)\to\infty$ as $t\to\infty$. Then for every $M>0$, there is a $T(M)>0$ such that for $t\geq T(M)$, $A(t)\geq M$. Now, for $t\geq T(M)$, we have 
\[
\int_0^t r^2(t-s)A(s)\,ds \geq \int_T^t r^2(t-s)\,ds \cdot M
=\int_{0}^{t-T} r^2(u)\,du \cdot M.
\]  
Since $r(0)=1$ and $r$ is continuous, $\int_0^\infty r^2(s)\,ds >0$, and we have
\[
B\geq \liminf_{t\to\infty} \int_0^t r^2(t-s)A(s)\,ds \geq 
\int_{0}^\infty r^2(u)\,du \cdot M.
\] 
Since $M>0$ is arbitrary, we may let $M\to\infty$, which gives $+\infty > B=\infty$, a contradiction. Therefore, we must have that $A$ tends to a finite limit, which is nothing other than
\[
 \int_0^\infty \left(g(s)+G(x_s)\right)^2 ds < \infty.
\]
We notice now in the case when $\mu$ is zero that $G(x_\cdot)=0$, so we have $g\in L^2(\mathbb{R_+})$ as required. 

The rest of the proof is devoted to the case when $\mu$ is non-trivial. By the above integrability, there exists a $C>0$, independent of $T>0$, such that 
\[
 \int_0^Tg^2(s)ds+2 \int_0^Tg(s)G(x_s)ds+ \int_0^TG^2(x_s)ds \leq C, \quad T >0.
\]
This inequality then implies
\[
\int_0^Tg^2(s)ds-\left|\int_{0}^T2g(s)G(x_s)ds\right| \leq \int_{0}^Tg^2(s)ds+\int_{0}^T2g(s)G(x_s)ds \leq C,
\]
for all $T >0$. By the Cauchy--Schwarz inequality, we have 
\begin{align*}
    \left| \int_{0}^T2g(s)G(x_s)ds \right|^2  \leq 4\int_{0}^Tg^2(s)ds \cdot \int_{0}^T G^2(x_s)ds 
\end{align*}
Next, recall that we have shown that $x\in L^2(\mathbb{R}_+)$ in Step 1 above. Since $\mu$ is a finite measure, we have that $G(x_\cdot) \in L^2(\mathbb{R}_+)$ also. Therefore, there exists $C'>0$ such that 
\[
\left| \int_{0}^T2g(s)G(x_s)ds \right|^2  \leq C'\int_{0}^Tg^2(s)ds, \quad T>0,
\]
which leads to
\[
\int_0^Tg^2(s)ds \leq C + \sqrt{C'\int_{0}^Tg^2(s)ds}, \quad T\geq 0.
\]
Define $p(x) \coloneqq x^2-\sqrt{C'}x-C$ for $x\in \mathbb{R}$. The inequality $p(x)\leq 0$ is satisfied for  $x\geq 0$ provided
\[
0 \leq x \leq \dfrac{\sqrt{C'}+\sqrt{C'+4C}}{2}.
\]
Therefore, as  $p\left(\sqrt{\int_{0}^Tg^2(s)ds}\right)\leq 0$,
	we have
\[
\int_{0}^Tg^2(s)ds \leq \left[ \dfrac{\sqrt{C'}+\sqrt{C'+4C}}{2} \right]^{2}, \quad T>0,
\]
which implies that $g\in L^2(\mathbb{R}_+)$, since the bound on the right hand side is independent of $T$. Hence we have shown condition \textbf{(A)}(iv), as required.\\
\newline
Step 4: $(\textbf{B}) \implies$ \textbf{(A)}(ii):\\
\newline
Suppose finally, by way of contradiction, that $\Vert G(r_{\cdot})\Vert_{L^2(\mathbb{R}_+)}^2 \geq 1$. Take $\psi$ deterministic and arbitrary. We know $\mathbb{E}[X^2] \in L^1(\mathbb{R}_+)$ forces both $Z$ and $\gamma$ to be  in $L^1(\mathbb{R}_+)$. Thus integrating equation \eqref{eq. Z in terms of gamma}, using Fubini's theorem and the fact $r(t)=0$ for all $t <0$ we see that
\begin{equation}\label{eq.Zintcontra}
\int_0^\infty Z(s)ds = \int_0^\infty \gamma(s)ds+\int_0^\infty G^2(r_u)du\cdot \int_0^\infty Z(s)ds.
\end{equation}
If $\Vert G(r_{\cdot})\Vert_{L^2(\mathbb{R}_+)}^2=1$, this implies $\int_0^\infty \gamma(s)ds=0$. Recall the argument from the proof of Theorem~\ref{thm. mean square X to zero} which showed one can always choose a deterministic $\psi$ such that there exists an non-trivial interval where $\gamma(t)$ is non zero, provided $\mu$ is not the zero measure (which is impossible in this case, since we have assumed $\Vert G(r_{\cdot})\Vert_{L^2(\mathbb{R}_+)}^2>0$). Choosing such a $\psi$ means that $\int_0^\infty \gamma(s)ds \neq 0$, forcing a contradiction.\\
If $b^2:=\Vert G(r_{\cdot})\Vert_{L^2(\mathbb{R}_+)}^2>1$ once again choose a $\psi$ such that $\gamma$ is strictly positive
on a non--trivial interval. Then $\int_0^\infty \gamma(s)\,ds >0$. Moreover, since $Z(t)\geq \gamma(t)$ we have that 
$\int_0^\infty Z(s)\,ds >0$. Therefore from \eqref{eq.Zintcontra} we get 
\[
\int_0^\infty Z(s)ds > b\int_0^\infty Z(s)\,ds,
\]
and since $b>1$, we arrive again at a contradiction.  
Thus we must have $\Vert G(r_{\cdot})\Vert_{L^2(\mathbb{R}_+)}^2<1$, and the theorem is proven.
\end{proof}

\textbf{Acknowledgements} The authors wish to thank the organisers of the FAATNA meeting, and especially Profs. Conte, Diogo and Messina, who organised the special session on ``Recent Advances in the Analysis and Numerical Solution of Evolutionary Integral Equations'', at which an early version of this work was presented. \\
The authors also wish to acknowledge the careful reading of our manuscript by the anonymous referee. 
\\
EL is supported by Science Foundation Ireland (16/IA/4443).
\\
JA is partially supported by the RSE Saltire Facilitation Network
on Stochastic Differential Equations: Theory, Numerics and Applications (RSE1832). 

\end{document}